\documentclass[juq]{siamonline1116}

\usepackage{lipsum}
\usepackage{amsfonts}
\usepackage{graphicx}
\usepackage{epstopdf}
\usepackage{algorithmic}
\ifpdf
  \DeclareGraphicsExtensions{.eps,.pdf,.png,.jpg}
\else
  \DeclareGraphicsExtensions{.eps}
\fi

\graphicspath{{./}{./figs/}}

  \def\clap#1{\hbox to 0pt{\hss#1\hss}}

\providecommand{\mat}[1]{\bm{#1}}%
\renewcommand{\vec}[1]{\mathbf{#1}}
\newcommand{\vecalt}[1]{\bm{#1}}

\providecommand{\mA}{\ensuremath{\mat{A}}}
\providecommand{\mB}{\ensuremath{\mat{B}}}
\providecommand{\mC}{\ensuremath{\mat{C}}}

\providecommand{\mE}{\ensuremath{\mat{E}}}

\providecommand{\mI}{\ensuremath{\mat{I}}}

\providecommand{\mW}{\ensuremath{\mat{W}}}

\providecommand{\mLambda}{\ensuremath{\mat{\Lambda}}}
\providecommand{\mSigma}{\ensuremath{\mat{\Sigma}}}

\providecommand{\mzero}{\ensuremath{\mat{0}}}

\providecommand{\vb}{\ensuremath{\vec{b}}}
\providecommand{\vc}{\ensuremath{\vec{c}}}

\providecommand{\vk}{\ensuremath{\vec{k}}}

\providecommand{\vw}{\ensuremath{\vec{w}}}
\providecommand{\vx}{\ensuremath{\vec{x}}}

\providecommand{\vmu}{\ensuremath{\vecalt{\mu}}}

\newcommand{\hvw}{\hat{\vw}}

\newcommand{\sI}{\mathcal{I}}

\newcommand{\sK}{\mathcal{K}}

\newcommand{\sO}{\mathcal{O}}

\newcommand{\sS}{\mathcal{S}}

\newcommand{\Exp}[1]{\mathbb{E}\left[#1\right]}

\newcommand{\Cov}[1]{\mathbb{C}\operatorname{ov}\left[#1\right]}

\newcommand{\Var}[1]{\operatorname{Var}\left[#1\right]}

\newcommand{\Prob}[1]{\mathbb{P}[#1]}
\newcommand{\indep}{\rotatebox[origin=c]{90}{$\models$}}
\newcommand{\bmat}[1]{\begin{bmatrix}#1\end{bmatrix}}

\newcommand{\dist}[2]{\mathrm{dist}\,\left(\colspan{#1},\,\colspan{#2}\right)}

\newcommand{\colspan}[1]{\operatorname{colspan}(#1)}

\newcommand{\MSE}[1]{\operatorname{MSE}\left[#1\right]}
\newcommand{\Bias}[1]{\operatorname{Bias}\left[#1\right]}

\newcommand{\SDRS}{\mathcal{S}_{\operatorname{DRS}}}

\newcommand{\CIR}{\mC_{\operatorname{IR}}}

\newcommand{\CSIR}{\mC_{\operatorname{SIR}}}
\newcommand{\hCSIR}{\hat{\mC}_{\operatorname{SIR}}}
\newcommand{\CAVE}{\mC_{\operatorname{AVE}}}

\newcommand{\CSAVE}{\mC_{\operatorname{SAVE}}}
\newcommand{\hCSAVE}{\hat{\mC}_{\operatorname{SAVE}}}

\newcommand{\Bind}{B_{\text{ind}}}
\newcommand{\Nrmin}{N_{r_{\min}}}

\usepackage{amsmath,mathtools}
\usepackage{tabulary,booktabs}
\usepackage{thmtools,thm-restate}
\usepackage{algorithm}
\usepackage{subfig}
\usepackage{bm}
\usepackage{color}
\usepackage{hyperref}

\newtheorem{prob}{Problem}
\newtheorem{assume}{Assumption}
\newtheorem{ex}{Example}

\numberwithin{theorem}{section}

\newcommand{\TheTitle}{Inverse regression for ridge recovery II:  Numerics} 
\newcommand{\TheAuthors}{A. Glaws, P. G. Constantine, and R. D. Cook}

\headers{\TheTitle}{\TheAuthors}

\title{{\TheTitle}\thanks{Submitted to the editors February 1, 2018.
\funding{The first author's work was supported by the Ben L.~Fryrear Ph.D.~Fellowship in Computational Science and the DARPA DSO Enabling Quantification of Uncertainty in Physical Systems program. The second author's work was supported by the DARPA DSO Enabling Quantification of Uncertainty in Physical Systems program and the U.S. Department of Energy Office of Science, Office of Advanced Scientific Computing Research, Applied Mathematics program under Award Number DE-SC-0011077.}}}

\author{
Andrew Glaws%
\thanks{Department of Computer Science, University of Colorado, Boulder, CO (\email{andrew.glaws@colorado.edu}).}
\and
Paul G.~Constantine%
\thanks{Department of Computer Science, University of Colorado, Boulder, CO
(\email{paul.constantine@colorado.edu}).}
\and
R.~Dennis Cook%
\thanks{Department of Applied Statistics, University of Minnesota, St. Paul, MN (\email{dennis@stat.umn.edu}, \url{http://users.stat.umn.edu/\string~rdcook/}).}
}

\usepackage{amsopn}

\ifpdf
\hypersetup{
  pdftitle={\TheTitle},
  pdfauthor={\TheAuthors}
}
\fi

\begin{document}

\maketitle

\begin{abstract}
\emph{The content of the following paper has been combined with 'Inverse regression for ridge recovery: A data-driven approach for parameter space dimension reduction in computational science models' (\url{https://arxiv.org/abs/1702.02227}).}

We investigate the application of sufficient dimension reduction (SDR) to a noiseless data set derived from a deterministic function of several variables. In this context, SDR provides a framework for \emph{ridge recovery}. In this second part, we explore the numerical subtleties associated with using two \emph{inverse regression} methods---sliced inverse regression (SIR) and sliced average variance estimation (SAVE)---for ridge recovery. This includes a detailed numerical analysis of the eigenvalues of the resulting matrices and the subspaces spanned by their columns. After this analysis, we demonstrate the methods on several numerical test problems.
\end{abstract}



\section{Introduction}
\label{sec:intro}

This paper is the second part of an investigation into the viability and applicability of inverse regression---a class of techniques for sufficient dimension reduction (SDR) in statistical regression~\cite{Cook98}---as a dimension reduction tool for high-dimensional problems that arise in uncertainty quantification (UQ)~\cite{SmithUQ2013,SullivanUQ2015,HandbookUQ}. Part I~\cite{Glaws17a} translates the theory of SDR and inverse regression to the context of deterministic function approximation, which is a more appropriate context than regression for representing computer simulations of physical systems found in UQ---as is common in the statistical context of \emph{computer experiments}~\cite{Sacks89,Koehler96,dace2003}. In particular, we show how the population matrices from sliced inverse regression (SIR)~\cite{Li91} and sliced average variance estimation (SAVE)~\cite{Cook91}---whose eigenspaces define dimension reduction subspaces---can each be interpreted as a matrix of integrals. In this part, we interpret the SIR and SAVE algorithms, applied in the approximation context, as numerical methods for estimating the integrals. This interpretation allows us to apply error analysis for numerical integration to study convergence of SIR and SAVE for dimension reduction. 

To our knowledge, this is the first work to (i) interpret the moments of the inverse regression as integrals and (ii) study the inverse regression algorithms as numerical approximation methods for integrals. Such interpretation and study are enabled by interpreting the data for SIR and SAVE as coming from random point queries of a deterministic function as opposed to the standard statistical interpretation of independent samples from a joint predictor/response density; see Part I~\cite{Glaws17a}. Convergence analysis for SIR and SAVE in the SDR literature follows standard statistical arguments to show \emph{root-$N$ consistency} of the estimators; see Li~\cite{Li91} for SIR and Cook for SAVE~\cite{Cook00a}. More precisely, the statistical estimates converge in probability to the true but unknown parameters with a rate of $N^{-1/2}$ as the number $N$ of predictor/response pairs increases; see Definition 8.1 in Lehmann and Casella~\cite{Lehmann1998}. One could rightly argue that our analysis in this paper is just mean-squared convergence of the same estimators in a slightly different context; see, e.g., Theorem 8.2 in Lehmann and Casella~\cite{Lehmann1998}. In fact, the rates of convergence for numerical approximation that we show for the SIR and SAVE subspaces are similarly (and not surprisingly) $\mathcal{O}(N^{1/2})$. However, one novel result from our analysis that, to our knowledge, has not appeared in the SDR literature is that the error in the estimated subspaces depends inversely on the associated spectral gap in the true matrix of integrals; this result employs arguments from stochastic eigenspace perturbation theory~\cite{Stewart90} and gives insight into the accuracy of the algorithms. 

Beyond this insight, however, what does the interpretation of SIR and SAVE as numerical integration provide? Importantly, it suggests that, for certain scenarios, there may be more efficient (i.e., same accuracy with fewer function evaluations) numerical methods for estimating the matrices of integrals. For example, we explore Gauss quadrature methods for inverse regression in~\cite{Glaws17c}. The current manuscript provides the theoretical basis for such explorations.

The rest of this paper is structured as follows. Section \ref{sec:SDR} briefly reviews SDR theory from Part I and introduces the SIR and SAVE algorithms for dimension reduction in statistical regression. Section \ref{sec:ridgerecovery} reviews the ridge recovery problem and contains a detailed numerical analysis of SIR and SAVE for ridge recovery. In Section \ref{sec:numericalresults} we apply SIR and SAVE for ridge recovery to three test problems. The first two are quadratic functions of 10 variables and the third is a simplified model of magnetohydrodynamics with five input parameters.

\section{Inverse regression and sufficient dimension reduction}
\label{sec:SDR}

The regression problem begins with predictor/response pairs $\{ [ \, \vx_i^T \, , \, y_i \, ] \}$, $i=1,\dots,N$, where $y_i \in \mathbb{R}$ and $\vx_i \in \mathbb{R}^m$, that are assumed to be realizations of the random vector $[ \, \vx^T \, , \, y \, ]$ with unknown joint density $\pi_{\vx,y}$. The goal of the regression problem is to statistically characterize the conditional random variable $y|\vx$. SDR attempts to reduce the dimension of the predictor space using a \emph{dimension reduction subspace}~\cite[Chapter 6]{Cook98}. If $\mA \in \mathbb{R}^{m \times n}$, $n \leq m$ is such that 
\begin{equation}
\label{eq:cond_indep_DRS}
y \, \indep \, \vx | \mA^T \vx,
\end{equation}
then $\SDRS = \colspan{\mA}$ is a dimension reduction subspace (DRS). The property \eqref{eq:cond_indep_DRS} is  \emph{conditional independence}. To obtain a well-posed SDR problem, we define the \emph{central subspace}, denoted by $\sS_{y|\vx}$, to be the DRS such that
\begin{equation}
\sS_{y|\vx} \subseteq \SDRS
\end{equation} 
for any other DRS $\SDRS$, provided that such a subspace exists~\cite{Cook96}. In this sense, the central subspace is the smallest DRS exhibited by a regression problem. Existence of the central subspace can be guaranteed under a variety of conditions~\cite[Chapter 6]{Cook98}. For our purposes, we consider an existence condition related to the marginal density of the predictors $\pi_{\vx}$. In short, if the support of $\pi_{\vx}$ is convex, then the associated regression problem exhibits a central subspace.

In~\cite{Glaws17a}, we introduced the weak assumption of standardized predictors, which simplifies discussion of SIR and SAVE. We restate this assumption here for reference, and unless otherwise noted we assume standardized predictors in the remainder.
\begin{assume}[Standardized predictors]
\label{assume:standardized_predictors}
Assume that $\vx$ is standardized so that 
\begin{equation}
\Exp{\vx} = \mzero \quad \text{and} \quad \Cov{\vx} = \mI .
\end{equation}
\end{assume}
The \emph{SDR problem} can be stated as estimating the central subspace for a regression problem from the given predictor/response pairs. We restate the SDR problem for reference.
\begin{prob}[SDR problem]
\label{prob:SDR}
Given predictor/response pairs $\{ [ \, \vx_i^T \, , \, y_i \, ] \}$, with $i=1,\dots,N$, assumed to be independent draws from a joint density $\pi_{\vx,y}$, compute a basis $\mA \in \mathbb{R}^{m \times n}$ for the central subspace $\sS_{y|\vx}$ of the random variable $y|\vx$.
\end{prob}
Before we examine the SIR and SAVE algorithms for Problem \ref{prob:SDR}, we discuss an important property of the central subspace. Consider applying a function $h: \mathbb{R} \rightarrow \mathbb{R}$ to the regression responses $\{y_i\}$ to produce a new regression problem. That is, we now have predictor/response pairs of the form
\begin{equation}
\{ [ \, \vx_i^T \, , \, h(y_i) \, ] \} , \qquad i=1,\dots,N,
\end{equation}
and we wish to study the conditional random variable $h(y)|\vx$. The central subspace associated with the new regression problem is contained within the original central subspace,
\begin{equation}
\label{eq:cent_sub_mapped_response}
\sS_{h(y)|\vx}
\;\subseteq\;
\sS_{y|\vx} ,
\end{equation}
with equality holding when $h$ is strictly monotone. The following example from~\cite[Chapter 6]{Cook98} shows that these subspaces need not always be equal.
\begin{ex}
Let $y|\vx$ be normally distributed with $\Exp{y|\vx} = 0$ and $\Var{y|\vx} = \Var{y| \mA^T \vx}$ where $\mA \in \mathbb{R}^{m \times n}$ contains the basis for the central subspace for $y|\vx$. Assume that $h(y)$ maps negative values of $y$ to $h_1\in\mathbb{R}$ and non-negative values of $y$ to $h_2\in\mathbb{R}$. Then, $\Prob{h(y) = h_1 | \vx} = \Prob{h(y) = h_1} = 1/2$ for all $\vx$, which implies that $\sS_{h(y)|\vx} = \{ \mzero \}$. However, $\sS_{y|\vx} \neq \{ \mzero \}$ in general. Thus, $\sS_{h(y)|\vx}$ is a strict subset of $\sS_{y|\vx}$.
\end{ex}

Next we examine the SIR and SAVE algorithms for the SDR problem \ref{prob:SDR}. These algorithms use the given data to estimate population matrices related to the $m$-dimensional conditional random variable $\vx|y$, referred to as the \emph{inverse regression}. Using $\vx|y$ to perform dimension reduction on $y|\vx$ is main feature of the SIR and SAVE algorithms. The inverse regression allows us to consider $m$ one-dimensional regression problems rather than one $m$-dimensional regression problem.

\subsection{Sliced inverse regression}
\label{subsec:SIR}

The population matrix of SIR is
\begin{equation}
\label{eq:CIR_reg}
\CIR
\;=\;
\Cov{\Exp{\vx|y}} ,
\end{equation} 
where the notation ``$\CIR$'' emphasizes that this population matrix does not include the \emph{slicing} component of SIR. Under the linearity condition (Condition 1 in~\cite{Glaws17a}), the column space of $\CIR$ is contained within the central subspace,
\begin{equation}
\label{eq:CIR_in_CentSub}
\colspan{\CIR}
\;\subseteq\;
\sS_{y|\vx} .
\end{equation}
The linearity condition can be difficult to verify. However, a generalization of it is satisfied when the marginal density of the predictors $\pi_{\vx}$ is elliptically symmetric---e.g., when $\pi_{\vx}$ is a multivariate Gaussian density. We assume elliptic symmetry holds for the discussion in this section.

We want to approximate $\CIR$ using the predictor/response pairs given in the regression problem. The difficulty in this approximation lies in computing sample estimates of $\Exp{\vx|y}$ for point values of the response $y$. To address this issue, SIR applies a sliced mapping of the response as follows. Partition the observed response space,
\begin{equation}
\label{eq:slices}
y_{\min} = \tilde{y}_0 < \tilde{y}_1 < \dots < \tilde{y}_{R-1} < \tilde{y}_{R} = y_{\max},
\end{equation}
where
\begin{equation}
y_{\min} = \min_{1 \leq i \leq N} y_i \quad \text{and} \quad y_{\max} = \max_{1 \leq i \leq N} y_i .
\end{equation}
For $r = 1,\dots,R$, let $J_r = [\tilde{y}_{r-1}, \tilde{y}_r]$ denote the $r$th slice of the partition, and define the function
\begin{equation}
\label{eq:hy_r}
h(y) \;=\; r
\quad \mbox{for} \quad
y \in J_r .
\end{equation}
We are interested in solving the SDR problem \ref{prob:SDR} for the conditional random variable $h(y)|\vx$. From \eqref{eq:cent_sub_mapped_response}, we know that this central subspace is contained within $\sS_{y|\vx}$. 

The SIR algorithm computes a sample estimate of the population matrix
\begin{equation}
\CSIR
\;=\;
\Cov{\Exp{\vx|h(y)}} ,
\end{equation}
where the notation ``$\CSIR$'' indicates that this matrix is with respect to the sliced version of the original regression problem. From \eqref{eq:CIR_in_CentSub}, we know that the column space of $\CSIR$ is contained with the central subspace $\sS_{h(y)|\vx}$ for the associated regression problem. Combining this with \eqref{eq:cent_sub_mapped_response} gives
\begin{equation}
\label{eq:SIR_sub_contain}
\colspan{\CSIR}
\;\subseteq\;
\sS_{h(y)|\vx} 
\;\subseteq\;
\sS_{y|\vx} .
\end{equation}
The sliced partition of the response space and the sliced mapping $h$ effectively bin the samples of the response, which enables computation of sample estimates of $\Exp{\vx|h(y)}$ over each slice. This is the essential idea behind the SIR algorithm provided in Algorithm \ref{alg:SIR}.
\begin{algorithm}
\caption{Sliced inverse regression \cite{Li91}} \label{alg:SIR}
\textbf{Given:} $N$ samples $\{ [ \, \vx_i^T \, , \, y_i \, ] \}$, $i = 1,\dots, N$, drawn independently according to $\pi_{\vx,y}$, and an integer $n$. \\
\textbf{Assumptions:} Assumption~\ref{assume:standardized_predictors} holds and the marginal density $\pi_{\vx}$ is elliptically symmetric.
\begin{enumerate}
\item Define a partition of the response space as in \eqref{eq:slices}, and let $J_r = [\tilde{y}_{r-1}, \tilde{y}_r]$ for $r = 1,\dots,R$. Let $\sI_r \subset \{ 1,\dots,N \}$ be the set of indices for which $y_i \in J_r$ and let $N_r=|\sI_r|$.
\item For $r = 1,\dots,R$, compute the sample mean $\hat{\vmu}_r$ of the predictors whose associated responses are in $J_r$,
\begin{equation} \label{eq:sample_slice_average}
\hat{\vmu}_r = \frac{1}{N_r} \sum_{i \in \sI_r} \vx_i .
\end{equation}
\item Compute the weighted sample covariance matrix
\begin{equation}
\label{eq:hDSIR}
\hCSIR \;=\; \frac{1}{N} \sum_{r=1}^R N_r \, \hat{\vmu}_r \, \hat{\vmu}_r^T .
\end{equation}
\item Compute the eigendecomposition,
\begin{equation}
\hCSIR \;=\; \hat{\mW} \hat{\mLambda} \hat{\mW}^T,
\end{equation}
where the eigenvalues are in descending order $\hat{\lambda}_1 \geq \hat{\lambda}_2 \geq \dots \geq \hat{\lambda}_m \geq 0$ and the eigenvectors are orthonormal.
\item Let $\hat{\mA} \in \mathbb{R}^{m \times n}$ be the first $n$ eigenvectors of $\hCSIR$.
\end{enumerate}
\end{algorithm} 

Eigenvectors of $\CSIR$ associated with nonzero eigenvalues provide a basis for the SIR subspace, $\colspan{\CSIR}$. If the approximated eigenvalues $\hat{\lambda}_{n+1},\dots,\hat{\lambda}_m$ from Algorithm \ref{alg:SIR} are small, then $m \times n$ matrix $\hat{\mA}$ approximates a basis for this subspace. However, determining the appropriate $n$ for dimension reduction can be difficult in practice; to sidestep this issue, we assume $n$ is an input to SIR. Li~\cite{Li91} and Cook~\cite{Cook91} propose significance tests based on the distribution of the average of the $m-n$ trailing estimated eigenvalues. These testing methods extend to the SAVE algorithm; see section \ref{subsec:SAVE}.

For a fixed number of slices, SIR has been shown to be $N^{-1/2}$-consistent for estimating $\colspan{\CSIR}$~\cite{Li91}. In principle, increasing the number of slices may provide improved estimation of the central DRS. However, in practice, the success of Algorithm~\ref{alg:SIR} has been found to be relatively insensitive to the number of slices. A good heuristic is to choose the number of slices such that each slice contains enough samples to estimate the conditional expectation accurately. For this reason, Li~\cite{Li91} suggests constructing slices such that the response samples are distributed nearly equally.

\subsection{Sliced average variance estimation}
\label{subsec:SAVE}

The population matrix of SAVE is
\begin{equation}
\label{eq:CAVE_reg}
\CAVE
\;=\;
\Exp{\left( \mI - \Cov{\vx|y} \right)^2} .
\end{equation}
When the marginal density $\pi_{\vx}$ is elliptically symmetric,  the column space of $\CAVE$ is contained within the regression's central subspace,
\begin{equation}
\label{eq:CAVE_in_CentSub}
\colspan{\CAVE}
\;\subseteq\;
\sS_{y|\vx} .
\end{equation}
We want to estimate $\CAVE$ using the given predictor response pairs. Similar to section \ref{subsec:SIR}, SAVE introduces slicing to enable this estimation. Let $J_r$, $r = 1,\dots,R$ denote a sliced partition from \eqref{eq:slices}, and define $h(y)$ as in \eqref{eq:hy_r}. Considering dimension reduction with respect to $h(y)|\vx$ results in
\begin{equation}
\label{eq:CSAVE}
\CSAVE
\;=\;
\Exp{\left( \mI - \Cov{\vx|h(y)} \right)^2} .
\end{equation}
The notation ``$\CSAVE$'' indicates that we are considering the sliced counterpart of the original regression problem. Combining \eqref{eq:CAVE_in_CentSub} and \eqref{eq:cent_sub_mapped_response},
\begin{equation}
\colspan{\CSAVE}
\;\subseteq\;
\sS_{h(y)|\vx}
\;\subseteq\;
\sS_{y|\vx} .
\end{equation}
Algorithm \ref{alg:SAVE} contains the SAVE algorithm, which computes the column span for the sample approximation $\hCSAVE$. This basis has been shown to be $N^{-1/2}$ consistent for a basis of the SAVE subspace, $\colspan{\CSAVE}$~\cite{Cook00a}. Increasing the number of slices improves the estimate but suffers the same drawbacks as in SIR. In practice, SAVE has been shown to perform poorly compared to SIR when relatively few predictor/response pairs are available. This is due to difficulties approximating the covariance within the slices with few samples. For this reason, Cook~\cite{Cook09} suggests trying both methods to approximate the central DRS.

\begin{algorithm}[]
\caption{Sliced average variance estimation \cite{Cook00a}} \label{alg:SAVE}
\textbf{Given:} $N$ samples $\{ [ \, \vx_i^T \, , \, y_i \, ] \}$, $i = 1,\dots, N$, drawn independently according to $\pi_{\vx,y}$ and an integer $n$. \\
\textbf{Assumptions:} Assumption~\ref{assume:standardized_predictors} holds and the marginal density $\pi_{\vx}$ is elliptically symmetric.
\begin{enumerate}
\item Define a partition of the response space as in \eqref{eq:slices}, and let $J_r = [\tilde{y}_{r-1}, \tilde{y}_r]$ for $r = 1,\dots,R$. Let $\sI_r \subset \{ 1,\dots,N \}$ be the set of indices $i$ for which $y_i \in J_r$ and let $N_r=|\sI_r|$.
\item For $r = 1,\dots,R$,
\begin{enumerate}
\item Compute the sample mean $\hat{\vmu}_r$ of the predictors whose associated responses are in the $J_r$,
\begin{equation}
\hat{\vmu}_r = \frac{1}{N_r} \sum_{i \in \sI_r} \vx_i .
\end{equation}
\item Compute the sample covariance $\hat{\mSigma}_r$ of the predictors whose associated responses are in $J_r$,
\begin{equation}
\hat{\mSigma}_r = \frac{1}{N_r - 1} \sum_{i \in \sI_r} \left( \vx_i - \hat{\vmu}_r \right) \left( \vx_i - \hat{\vmu}_r \right)^T
\end{equation}
\end{enumerate}
\item Compute the matrix,
\begin{equation}
\label{eq:hDSAVE}
\hCSAVE \;=\; \frac{1}{N} \sum_{r=1}^R N_r \left( \mI - \hat{\mSigma}_r \right)^2 .
\end{equation}
\item Compute the eigendecomposition,
\begin{equation}
\hCSAVE \;=\; \hat{\mW} \hat{\mLambda} \hat{\mW}^T,
\end{equation}
where the eigenvalues are in descending order $\hat{\lambda}_1 \geq \hat{\lambda}_2 \geq \dots \geq \hat{\lambda}_m \geq 0$ and the eigenvectors are orthonormal.
\item Let $\hat{\mA} \in \mathbb{R}^{m \times n}$ be the first $n$ eigenvectors of $\hCSAVE$.
\end{enumerate}
\end{algorithm}

SIR and SAVE look for dimension reduction in regression problems where the predictor/response pairs are assumed to be independent draws from an unknown joint density. In the next section, we consider SIR and SAVE as methods for ridge recovery in the context of deterministic function approximation, where we seek to understand the numerical behavior of these algorithms.

\section{Ridge recovery}
\label{sec:ridgerecovery}

In this section, we consider functions of the form
\begin{equation}
\label{eq:y_fx}
y \;=\; f(\vx) ,
\qquad
\vx \in \mathbb{R}^m ,
\quad
y \in \mathbb{R} ,
\end{equation}
where $y$ is the scalar-valued output and $\vx$ is a vector of inputs. Additionally, we assume the input space is weighted by a known probability measure $\rho$. This setup is a typical mathematical description of the models that arise in computer experiments~\cite{Koehler96,Sacks89,dace2003}. Typical choices for $\rho$ in computer experiments include the multivariate Gaussian or the uniform measure over an $m$-dimensional rectangle defined by independent ranges of each component of $\vx$. We assume that $\rho$ has density function $p:\mathbb{R}^m \rightarrow \mathbb{R}^+$ and that the components of $\vx$ are independent. We note that the output space $\mathbb{R}$ is weighted by the push-forward probability measure $\gamma$, defined by $\rho$ and $f$.


The function $f: \mathbb{R}^m \rightarrow \mathbb{R}$ is a \emph{ridge function} if there exists $\mA \in \mathbb{R}^{m \times n}$ with $n < m$ such that
\begin{equation}
\label{eq:y_fx_gAx}
y
\;=\;
f(\vx)
\;=\;
g(\mA^T \vx) 
\end{equation}
for some $g: \mathbb{R}^n \rightarrow \mathbb{R}$~\cite{Pinkus15}. The column span of $\mA$ defines a DRS for $f$. Furthermore, Theorem 4.1 from~\cite{Glaws17a} shows that the central subspace from Section \ref{sec:SDR} can be extended to ridge functions. We denote the central subspace in this context by $\sS_{f,\rho}$ as it depends on both the function $f$ and the probability measure $\rho$. This subspace is the minimum dimension subspace that contains all the ridge information about $f$. The central subspace is guaranteed to exist for deterministic functions provided that the input density function $p$ has convex support. Since the input weighting is given, we can easily verify the existence of a central subspace.

The problem of discovering a basis for the central subspace of a ridge function is known as \emph{ridge recovery}.
\begin{prob}[Ridge recovery] \label{prob:ridge_recovery}
Given an input probability measure $\rho$ and a ridge function $f: \mathbb{R}^m \rightarrow \mathbb{R}$, compute a basis $\mA \in \mathbb{R}^{m \times n}$ for the central subspace $\sS_{f,\rho}$ using only point queries $(\vx, f(\vx))$.
\end{prob}
In the following sections, we examine SIR and SAVE as methods to solve or approximate the ridge recovery problem. Due to the deterministic nature of the problem, we interpret these algorithms as numerical approximation methods and study their convergence properties. 

Recall from section \ref{sec:SDR}, that SIR and SAVE employ the inverse regression $\vx|y$. This is a $m$-dimensional random vector parameterized by the scalar-valued response $y$. For deterministic functions, the analogous concept is the inverse image $f^{-1} (y) = \{ \, \vx \in \mathbb{R}^m \, : \, f(\vx) = y \, \}$. The inverse image is weighted by the conditional probability measure $\sigma_{\vx|y}$, defined as the restriction of $\rho$ to $f^{-1} (y)$~\cite{Chang1997}.

Before analyzing SIR and SAVE for ridge recovery, we introduce a lemma that we use in the proofs. The lemma allows us to quickly analyze the asymptotic behavior of sums of expectations over complicated multi-index sets. Define the multi-index set
\begin{equation} \label{eq:K_multi-index_set}
\sK^{p,q} (r)
\;=\;
\left\{ \vk \in \mathbb{N}^{p+q} \left| \begin{matrix} k_1 , \dots k_p \in \left\{ 1, \dots , N \right\} \\
\text{ and } k_{p+1} , \dots , k_{p+q} \in \left\{ 1 , \dots , N_r \right\} \end{matrix} \right. \right\} 
\end{equation}
where $k_i$ denotes the $i$th entry of $\vk$ and $N_r \leq N$. Let $\vb \in \mathbb{R}^{p+q}$ be a random vector and let $\left\{ \vb^k \right\}$, $k=1,\dots,N$, denote $N$ independently drawn realizations of $\vb$ where $N$ is from \eqref{eq:K_multi-index_set}. Define the $(p+q)$-dimensional tensor $\mB^{p,q} (r)$ whose elements are given by
\begin{equation} \label{eq:pq_tensor}
\mB^{p,q}_{\vk} (r)
\;=\;
\Exp{b_1^{k_1} \dots b_p^{k_p} b_{p+1}^{k_{p+1}} \dots b_{p+q}^{k_{p+q}}} 
\end{equation}
for $\vk \in \sK^{p,q} (r)$. Note that the $k_i$'s above do not indicate powers of $b_j$, but rather indices. That is, $b_j^k$ denotes the $j$th entry of $\vb^k$ which is the $k$th realization of the random vector $\vb$. Thus, $\mB^{p,q} (r)$ has dimensions
\begin{equation}
\mB^{p,q} (r)
\in
\mathbb{R}^{\overbrace{N \times \dots \times N}^{p} \times \overbrace{N_r \times \dots \times N_r}^{q}} .
\end{equation}

The following lemma plays an important role in upcoming proofs.
\begin{lemma}
\label{lem:general_B_sum}
Let $\vb \in \mathbb{R}^{p+q}$ be a random vector and let $\left\{ \vb^k \right\}$, $k=1,\dots,N$, denote $N$ independently drawn realizations of $\vb$. Define $N_r \in \mathbb{N}$ such that $p+q \leq N_r \leq N$. Let $\sK^{p,q} (r)$ and $\mB^{p,q} (r)$ be defined according to \eqref{eq:K_multi-index_set} and \eqref{eq:pq_tensor}, respectively. Then,
\begin{equation}
\sum_{\vk \in \sK^{p,q} (r)} \mB^{p,q}_{\vk} (r) \;=\; N^p N_r^q \, \Exp{b_1} \dots \Exp{b_p} \Exp{b_{p+1}} \dots \Exp{b_{p+q}} + \sO (N^p N_r^{q-1} + N^{p-1} N_r^q).
\end{equation}
\end{lemma}

\begin{proof}
Define the following subsets of $\sK^{p,q} (r)$:
\begin{equation}
\sK^{p,q}_1 (r)
\;=\;
\left\{
\vk \in \sK^{p,q} (r)
\left|
k_i \neq k_j \text{ for } i,j = 1, \dots , p+q \text{ and } i \neq j
\right.
\right\}
\end{equation}
and $\sK^{p,q}_2 (r) = \sK^{p,q} (r) \setminus \sK^{p,q}_1 (r)$. Asymptotically, as $N$ and $N_r$ tend towards infinity, the subset $\sK^{p,q}_1 (r)$ has
\begin{equation}
\left( \prod_{i=0}^{q-1} (N_r - i) \right)
\left( \prod_{j=q}^{p+q-1} (N - j) \right)
\;=\;
N^p N_r^q + \sO (N^p N_r^{q-1} + N^{p-1} N_r^q)
\end{equation}
elements. Since all of the indices are unique for any $\vk \in \sK^{p,q}_1 (r)$,
\begin{equation}
\mB^{p,q}_{\vk} (r)
\;=\;
\Exp{b_1} \dots \Exp{b_p} \Exp{b_{p+1}} \dots \Exp{b_{p+q}} .
\end{equation}

Consider the number of elements in $\sK^{p,q}_2 (r)$. By construction, any $\vk \in \sK^{p,q}_2 (r)$ must have at least two identical indices. As $N$ and $N_r$ tend towards infinity, the largest subset of elements in $\sK^{p,q}_2 (r)$ will be those with only two identical elements. There are three such cases: \\
Case 1: \\
Two of the $k_i$'s which range from $1,\dots,N$ are identical and all other $k_i$'s are unique. There are 
\begin{equation}
\left( \prod_{i=0}^{q-1} (N_r - i) \right)
\left( \prod_{j=q}^{p+q-2} (N - j) \right)
\;=\;
\sO (N^{p-1} N_r^q)
\end{equation} 
such elements. \\
Case 2: \\
Two of the $k_i$'s which range from $1,\dots,N_r$ are identical and all other $k_i$'s are unique. There are 
\begin{equation}
\left( \prod_{i=0}^{q-2} (N_r - i) \right)
\left( \prod_{j=q-1}^{p+q-2} (N - j) \right)
\;=\;
\sO (N^p N_r^{q-1})
\end{equation} 
such elements. \\
Case 3: \\
One $k_i$ which range from $1,\dots,N$ and one $k_i$ which ranges from $1,\dots,N_r$ are identical and all other $k_i$'s are unique. There are 
\begin{equation}
\left( \prod_{i=0}^{q-1} (N_r - i) \right)
\left( \prod_{j=q}^{p+q-2} (N - j) \right)
\;=\;
\sO (N^{p-1} N_r^q)
\end{equation} 
such elements.
Thus, there are $\sO(N^p N_r^{q-1} + N^{p-1} N_r^q)$ elements in $\sK^{p,q}_2 (r)$.

Therefore, 
\begin{equation}
\begin{aligned}
\sum_{\vk \in \sK^{p,q} (r)} \mB^{p,q}_{\vk} (r)
\;&=\;
\sum_{\vk \in \sK^{p,q}_1 (r)} \mB^{p,q}_{\vk} (r) + \sum_{\vk \in \sK^{p,q}_2 (r)} \mB^{p,q}_{\vk} (r) \\
\;&=\;
\left( N^p N_r^q + \sO (N^p N_r^{q-1} + N^{p-1} N_r^q) \right) \Exp{b_1} \dots \Exp{b_p} \Exp{b_{p+1}} \dots \Exp{b_{p+q}} \\
&\hspace{150pt}
+ \sO (N^p N_r^{q-1} + N^{p-1} N_r^q) \\
\;&=\;
N^p N_r^q \, \Exp{b_1} \dots \Exp{b_p} \Exp{b_{p+1}} \dots \Exp{b_{p+q}} + \sO (N^p N_r^{q-1} + N^{p-1} N_r^q),
\end{aligned}
\end{equation}
as required.
\end{proof}

\subsection{SIR for ridge recovery}
\label{subsec:SIR_for_RR}

In this section, we consider SIR (Algorithm \ref{alg:SIR}) as a method for ridge recovery (Problem \ref{prob:ridge_recovery}). We rewrite \eqref{eq:CIR_reg} as an integral over the output space
\begin{equation}
\label{eq:CIR}
\CIR
\;=\;
\int \vmu (y) \, \vmu (y)^T \, d \gamma(y) ,
\end{equation}
where $\gamma$ is the push-forward measure of $\rho$ through $f$ from the previous section. The conditional expectation $\vmu(y)$ is
\begin{equation}
\label{eq:cond_exp}
\vmu(y)
\;=\;
\int \vx \, d \sigma_{\vx|y} (\vx) ,
\end{equation}
where $\sigma_{\vx|y}$ is the conditional measure defined on the inverse image $f^{-1} (y)$. Approximating this measure is difficult, especially if the structure of $f$ is unknown. SIR considers a sliced-mapping of $f$ which enables the approximation of \eqref{eq:CIR} and \eqref{eq:cond_exp}.

Recall the sliced partitioning of the output space from \eqref{eq:slices}. We apply the same partitioning to the output space of $f$ here. Let $h: \mathbb{R} \rightarrow \mathbb{R}$ be defined as in \eqref{eq:hy_r}. We now consider the ridge recovery problem for $r = h(y) = h(f(\vx))$, where $r \in \{ 1,\dots,R \}$. The range of $h$ is weighted by the probability mass function
\begin{equation}
\label{eq:slice_density}
\omega_r
\;=\;
\int_{J_r} \, d \gamma (y) , \qquad r=1,\dots,R.
\end{equation}
Without loss of generality, we assume that the slices $\{J_r\}$ are constructed such that $\omega_r > 0$ for all values of $r$. If $\omega_r = 0$ for some value of $r$, then we can combine this slice with one of the adjacent slices without altering the problem. 

We rewrite the conditional expectation from \eqref{eq:cond_exp} in terms of the sliced output as
\begin{equation}
\label{eq:cond_exp_r}
\vmu_r
\;=\;
\int \vx \, d \sigma_{\vx|r} (\vx) ,
\end{equation}
where $\sigma_{\vx|r}$ is the conditional measure defined over the set $f^{-1} ( h^{-1} (r) ) =  \{ \, \vx \in \mathbb{R}^m \, : \, h(f(\vx)) = r \, \}$. Using \eqref{eq:slice_density} and \eqref{eq:cond_exp_r}, we can write the slice-based version of $\CIR$ as
\begin{equation}
\label{eq:CSIR}
\CSIR
\;=\;
\sum_{r=1}^R \omega_r \, \vmu_r \, \vmu_r^T .
\end{equation}
By Theorem 4.1 from~\cite{Glaws17a}, we know that properties of the central subspace extend to the ridge recovery problem. This includes containment of the central subspace under any mapping of the output. Thus,
\begin{equation}
\label{eq:SIR_sub_contain_det}
\colspan{\CSIR}
\;\subseteq\;
\sS_{h,\rho}
\;\subseteq\;
\sS_{f,\rho} .
\end{equation}
By approximating $\CSIR$, we obtain an approximation of at least part of the central subspace. 

The containment property in \eqref{eq:SIR_sub_contain_det} provides some insight into the value of SIR for ridge recovery. However, the function approximation context requires a more rigorous understanding of how $\CSIR$ approximates $\CIR$. We note that $\CSIR$ contains a finite sum approximation of the integral in $\CIR$. By recognizing that $f^{-1} ( h^{-1} (r)) = \cup_{y \in J_r} f^{-1} (y)$, we see that $\vmu_r$ is the average value of the conditional expectations with respect to $y$ over all $y \in J_r$. That is,
\begin{equation}
\label{eq:vmur}
\vmu_r
\;=\;
\int_{J_r} \vmu (y) \, d \gamma (y), \qquad r=1,\dots,R.
\end{equation}
Therefore, $\CSIR$ approximates $\CIR$ by a weighted sum of the average values of $\vmu(y)$ within each slice. If $y=f(\vx)$ is continuous, then $\vmu(y)$ is continuous almost everywhere with respect to the push forward measure $\gamma$. Therefore, $\CIR$ is Riemann integrable~\cite[Ch. 2]{Folland99}. This ensures that sum approximations using the supremum and infimum of $\vmu(y)$ over each slice converge to the same value as the number $R$ of slices increases. By the sandwich theorem, the average value  converges as well~\cite{Abbott01}. Therefore, we may consider $\CSIR$ to be a Riemann sum approximation of $\CIR$. Furthermore, by letting the number of slices $R$ increase to infinity, we get that $\CSIR$ converges to $\CIR$.

Algorithm \ref{alg:SIR} uses predictor/response pairs $\{ [ \, \vx_i^T \, , \, y_i \, ] \}$ for $i = 1,\dots,N$ to compute $\hCSIR$. In the statistical regression context, these pairs are  independent realizations drawn according to an unknown joint density $\pi_{\vx,y}$. For deterministic functions, we construct a similar set of input/output pairs by sampling the input space according to $\rho$ and evaluating $y = f(\vx)$ at each point. We can then use Algorithm \ref{alg:SIR} for ridge recovery. In this context, we interpret the SIR results as a simple Monte Carlo approximation of the integrals \eqref{eq:vmur} in $\CSIR$. This results in the random matrix $\hCSIR$. Anything computed using $\hCSIR$---including eigenvalues and eigenvectors---is also random. For this reason, the convergence analysis for Algorithm \ref{alg:SIR} in the context of ridge recovery is probabilistic.

We next study the convergence properties of the SIR algorithm for ridge recovery. Recall from Algorithm \ref{alg:SIR} that $N_r$ denotes the number of samples in each slice $J_r$ for $r = 1, \dots, R$. The convergence depends on the smallest number of samples per slice over all the slices. We denote this value by
\begin{equation} \label{eq:N_r_min}
N_{r_{\min}}
\;=\;
\min_{1 \leq r \leq R} N_r .
\end{equation}
Also, recall that the slices are assumed to be constructed such that $\omega_r > 0$. Thus, $N_{r_{\min}} > 0$ with probability 1 as $N \rightarrow \infty$. 

The following theorem shows that the eigenvalues of $\hCSIR$ converge to those of $\CSIR$ in a mean-squared sense.
\begin{theorem}
\label{thm:SIR_eig_converge}
Assume that Algorithm \ref{alg:SIR} has been applied to the data set $\{ [ \, \vx_i^T \, , \, f(\vx_i) \, ] \}$, with $i=1,\dots,N$, where the $\vx_i$ are drawn independently according to $\rho$. Then, for $k = 1,\dots,m$,
\begin{equation}
\Exp{\left( \lambda_k (\CSIR) - \lambda_k (\hCSIR) \right)^2}
\;=\;
\sO (N^{-1}_{r_{\min}})
\end{equation}
where $\lambda_k (\cdot)$ denotes the $k$th eigenvalue of the given matrix.
\end{theorem}

\begin{proof}
Recall the SIR matrix from \eqref{eq:CSIR},
\begin{equation}
\label{eq:CSIR2}
\CSIR \;=\; \sum_{r=1}^R \omega_r \, \vmu_r \, \vmu_r^T
\end{equation}
where $\omega_r$ is the probability mass function from \eqref{eq:slice_density}. Algorithm \ref{alg:SIR} approximates $\CSIR$ using
\begin{equation}
\label{eq:hCSIR}
\hCSIR
\;=\;
\sum_{r=1}^R \hat{\omega}_r \, \hat{\vmu}_r \, \hat{\vmu}_r^T
\end{equation}
where $\hat{\omega}_r$ and $\hat{\vmu}_r$ are sample estimates of $\omega_r$ and $\vmu_r$, respectively,
\begin{equation}
\hat{\omega}_r
\;=\;
\frac{N_r}{N}
\qquad \text{and} \qquad
\hat{\vmu}_r
\;=\;
\frac{1}{N_r} \sum_{i \in \sI_r} \vx_i
\end{equation}
where $\sI_r$ is the set of indices for which $y_i \in J_r$ and $N_r$ is the cardinality of $\sI_r$. We may rewrite $\hat{\omega}_r$ as
\begin{equation}
\hat{\omega}_r
\;=\;
\frac{1}{N}  \sum_{i=1}^N \chi \left( y_i \in J_r \right)
\end{equation}
where $\chi \left( y_i \in J_r \right)$ is an indicator function that is 1 when $y_i \in J_r$ and 0 otherwise. 

Fix $r$ such that we are considering a single term of the summations in \eqref{eq:CSIR2} and \eqref{eq:hCSIR}. Let 
\begin{equation}
\omega \;=\; \omega_r , \quad
\hat{\omega} \;=\; \hat{\omega}_r , \quad
\chi_i \;=\; \chi \left( y_i \in J_r \right) , \quad
\vmu \;=\; \vmu_r , \quad \text{and} \quad 
\hat{\vmu} \;=\; \hat{\vmu}_r .
\end{equation}
Assume without loss of generality that $\sI_r = \{ 1 , \dots , N_r \}$. This simplifies notation as we compute the mean squared error.

To compute the mean squared error, we focus on the computation of a single element in $\hCSIR$. To do this, we use the following notation: let $x_i^k$ denote the $i$th element of the vector $\vx^k$ which is the $k$th realization of the random vector $\vx$. Thus, for $1 \leq i , j \leq m$,
\begin{equation} \label{eq:SIR_exp_of_prod1}
\begin{aligned}
\Exp{\hat{\omega} \, \hat{\mu}_i \, \hat{\mu}_j}
\;&=\;
\Exp{\left( \frac{1}{N} \sum_{k_1=1}^N \chi^{k_1} \right)
\left( \frac{1}{N_r} \sum_{k_2=1}^{N_r} x_i^{k_2} \right)
\left( \frac{1}{N_r} \sum_{k_3=1}^{N_r} x_j^{k_3} \right)} \\
\;&=\;
\frac{1}{N \, N_r^2} \sum_{k_1=1}^N \sum_{k_2=1}^{N_r} \sum_{k_3=1}^{N_r} \Exp{\chi^{k_1} \, x_i^{k_2} \, x_j^{k_3}} .
\end{aligned}
\end{equation}
Equations \eqref{eq:K_multi-index_set} and \eqref{eq:pq_tensor} allow us to rewrite \eqref{eq:SIR_exp_of_prod1} as a summation over a tensor. Since we assume $r$ is fixed, we drop the argument from the notation of these equations. Thus,
\begin{equation} \label{eq:tensor_sum}
\Exp{\hat{\omega} \, \hat{\mu}_i \, \hat{\mu}_j}
\;=\;
\frac{1}{N \, N_r^2} \sum_{\vk \in \sK^{1,2}} \mB^{1,2}_{\vk} .
\end{equation}
From Lemma \ref{lem:general_B_sum},
\begin{equation} \label{eq:SIR_exp_of_prod2}
\begin{aligned}
\Exp{\hat{\omega} \, \hat{\mu}_i \, \hat{\mu}_j}
\;&=\;
\frac{1}{N \, N_r^2} \sum_{\vk \in \sK^{1,2}} \mB^{1,2}_{\vk} \\
\;&=\;
\frac{1}{N \, N_r^2} \left[ N \, N_r^2 \, \Exp{\chi} \Exp{x_i} \Exp{x_j} + \sO (N \, N_r + N_r^2) \right] \\
\;&=\;
\Exp{\chi} \Exp{x_i} \Exp{x_j} + \sO (N_r^{-1}) \\
\;&=\;
\omega \, \mu_i \, \mu_j + \sO(N_r^{-1}) .
\end{aligned}
\end{equation}
Next, we compute the variance using the property
\begin{equation}
\Var{\hat{\omega} \, \hat{\mu}_i \, \hat{\mu}_j}
\;=\;
\Exp{ \left( \hat{\omega} \, \hat{\mu}_i \, \hat{\mu}_j \right)^2} 
- \Exp{\hat{\omega} \, \hat{\mu}_i \, \hat{\mu}_j}^2 .
\end{equation}
To find $\Exp{ \left( \hat{\omega} \, \hat{\mu}_i \, \hat{\mu}_j \right)^2}$,
\begin{equation}
\begin{aligned}
\Exp{ \left( \hat{\omega} \, \hat{\mu}_i \, \hat{\mu}_j \right)^2}
\;&=\;
\Exp{\left( \frac{1}{N} \sum_{k_1=1}^N \chi^{k_1} \right)^2
\left( \frac{1}{N_r} \sum_{k_2=1}^{N_r} x_i^{k_2} \right)^2
\left( \frac{1}{N_r} \sum_{k_3=1}^{N_r} x_j^{k_3} \right)^2} \\
\;&=\;
\frac{1}{N^2 N_r^4} \sum_{k_1=1}^N \sum_{k_2=1}^N \sum_{k_3=1}^{N_r} \sum_{k_4=1}^{N_r} \sum_{k_5=1}^{N_r} \sum_{k_6=1}^{N_r} \Exp{\chi^{k_1} \chi^{k_2} x_i^{k_3} x_i^{k_4} x_j^{k_5} x_j^{k_6}} .
\end{aligned}
\end{equation}
Again, we can express this summation using \eqref{eq:K_multi-index_set} and \eqref{eq:pq_tensor} as 
\begin{equation}
\Exp{\left( \hat{\omega} \, \hat{\mu}_i \, \hat{\mu}_j \right)^2}
\;=\;
\frac{1}{N^2 N_r^4} \sum_{\vk \in \sK^{2,4}} \mB^{2,4}_{\vk} .
\end{equation}
By Lemma \ref{lem:general_B_sum},
\begin{equation}
\begin{aligned}
\Exp{\left( \hat{\omega} \, \hat{\mu}_i \, \hat{\mu}_j \right)^2}
\;&=\;
\frac{1}{N^2 N_r^4} \sum_{\vk \in \sK^{2,4}} \mB^{2,4}_{\vk} \\
\;&=\;
\frac{1}{N^2 N_r^4} \left[ N^2 N_r^4 \, \Exp{\chi}^2 \Exp{x_i}^2 \Exp{x_j}^2 + \sO (N^2 N_r^3 + N N_r^4) \right] \\
\;&=\;
\Exp{\chi}^2 \Exp{x_i}^2 \Exp{x_j}^2 + \sO (N_r^{-1}) \\
\;&=\;
\omega^2 \mu_i^2 \mu_j^2 + \sO (N_r^{-1}) .
\end{aligned}
\end{equation}
Thus,
\begin{equation} \label{eq:SIR_var_of_prod}
\begin{aligned}
\Var{\hat{\omega} \, \hat{\mu}_i \, \hat{\mu}_j}
\;&=\;
\Exp{\left( \hat{\omega} \, \hat{\mu}_i \, \hat{\mu}_j \right)^2} - \Exp{\hat{\omega} \, \hat{\mu}_i \, \hat{\mu}_j}^2 \\
\;&=\;
\left( \omega^2 \mu_i^2 \mu_j^2 + \sO (N_r^{-1}) \right) - \left( \omega \, \mu_i \, \mu_j + \sO (N_r^{-1}) \right)^2 \\
\;&=\;
\sO (N_r^{-1}) .
\end{aligned}
\end{equation}
Equations \eqref{eq:SIR_exp_of_prod2} and \eqref{eq:SIR_var_of_prod} hold for each $r = 1,\dots,R$. Let $(\cdot)_{ij}$ denote the $ij$th element of the given matrix. Since $\hCSIR$ is a summation over $r=1,\dots,R$,
\begin{equation} \label{eq:SIR_exp_and_var}
\Exp{(\hCSIR)_{ij}}
\;=\;
(\CSIR)_{ij} + \sO \left( \Nrmin^{-1} \right)
\quad \text{and} \quad
\Var{(\hCSIR)_{ij}}
\;=\;
\sO \left( \Nrmin^{-1} \right) 
\end{equation}
where $\Nrmin$ is the same from \eqref{eq:N_r_min}, 
\begin{equation}
\Nrmin \;=\; \min_{1 \leq r \leq R} N_r .
\end{equation}
From \eqref{eq:SIR_exp_and_var}, the mean squared error for each element of $\hCSIR$ is
\begin{equation} \label{eq:SIR_MSE}
\begin{aligned}
\MSE{(\hCSIR)_{ij}}
\;&=\;
\Bias{(\hCSIR)_{ij}}^2 + \Var{(\hCSIR)_{ij}} \\
\;&=\;
\left( \Exp{(\hCSIR)_{ij}} - (\CSIR)_{ij} \right)^2
+ \Var{(\hCSIR)_{ij}} \\
\;&=\;
\left( (\CSIR)_{ij} + \sO \left( \Nrmin^{-1} \right) - (\CSIR)_{ij} \right)^2
+ \sO \left( \Nrmin^{-1} \right) \\
\;&=\;
\sO \left( \Nrmin^{-1} \right) .
\end{aligned}
\end{equation}
Next, we examine how the element-wise mean squared error in \eqref{eq:SIR_MSE} translates to errors in the eigenvalue estimates. By Corollary 8.1.6 in~\cite{Golub96},
\begin{equation}
\left| \lambda_k (\CSIR) - \lambda_k (\hCSIR) \right| \;\leq\; || \mE ||_2
\end{equation}
where $\mE = \CSIR - \hCSIR$. Since $||\cdot||_2 \leq ||\cdot||_F$,
\begin{equation}
\left| \lambda_k (\CSIR) - \lambda_k (\hCSIR) \right| \;\leq\; || \mE ||_F .
\end{equation}
Squaring both sides and taking the expectation,
\begin{equation}
\Exp{ \left( \lambda_k (\CSIR) - \lambda_k (\hCSIR) \right)^2} \;\leq\; \Exp{|| \mE ||_F^2} .
\end{equation}
Consider the right-hand side
\begin{equation} \label{eq:exp_E}
\begin{aligned}
\Exp{|| \mE ||_F^2}
\;&=\;
\Exp{\sum_{i=1}^m \sum_{j=1}^m (\mE)_{ij}^2} \\
\;&=\;
\sum_{i=1}^m \sum_{j=1}^m \Exp{(\mE)_{ij}^2} \\
\;&=\;
\sum_{i=1}^m \sum_{j=1}^m \Exp{ \left( (\CSIR)_{ij} - (\hCSIR)_{ij} \right)^2} \\
\;&=\;
\sum_{i=1}^m \sum_{j=1}^m \MSE{(\hCSIR)_{ij}} \\
\;&=\;
\sO (\Nrmin^{-1}) 
\end{aligned}
\end{equation}
as required.
\end{proof}

In words, the mean-squared error in the eigenvalues of $\hCSIR$ decays at a $N_{r_{\min}}^{-1}$ rate. Since $\omega_r > 0$ for all $r$, $N_{r_{\min}} \rightarrow \infty$ as $N \rightarrow \infty$. Moreover, the convergence rate suggests that one should define the slices in Algorithm \ref{alg:SIR} such that the same number of samples appears in each slice. This maximizes $N_{r_{\min}}$ and reduces the error in the eigenvalues. In practice, the eigenvalues are often used to determine the dimension of the approximated subspace, so understanding the approximation error is important.


The next theorem shows the convergence of the subspaces produced by Algorithm \ref{alg:SIR} as the number of point evaluations of $f$ increases. We measure convergence using the subspace distance~\cite{Golub96},
\begin{equation} \label{eq:sub_dist}
\dist{\mA}{\hat{\mA}} \;=\; 
\left\| \mA \mA^T - \hat{\mA} \hat{\mA}^T \right\|_2,
\end{equation}
where $\mA, \hat{\mA}$ are the first $n$ eigenvectors of $\CSIR$ and $\hCSIR$, respectively. The distance metric \eqref{eq:sub_dist} is the principal angle between the subspaces $\colspan{\mA}$ and $\colspan{\hat{\mA}}$. 
\begin{theorem}
\label{thm:SIR_sub_converge}
Assume the same conditions from Theorem \ref{thm:SIR_eig_converge}. Then, for sufficiently large $N$,
\begin{equation}
\dist{\mA}{\hat{\mA}} \;=\; \frac{1}{\lambda_n (\CSIR) - \lambda_{n+1} (\CSIR)}\;\sO (N^{-1/2}_{r_{\min}})
\end{equation}
with high probability.
\end{theorem}

\begin{proof}
The matrices $\mA , \hat{\mA} \in \mathbb{R}^{m \times n}$ contain the first $n$ eigenvectors of $\CSIR$ and $\hCSIR$, respectively. Let $\mB , \hat{\mB} \in \mathbb{R}^{m \times (m-n)}$ contain the last $m-n$ eigenvectors of each matrix. By Corollary 8.1.11 in~\cite{Golub96}, if 
\begin{equation}
||\mE||_2 \;\leq\; \frac{\lambda_n (\CSIR) - \lambda_{n+1} (\CSIR)}{5} ,
\end{equation}
then
\begin{equation} \label{eq:Cor8.1.11}
\dist{\mA}{\hat{\mA}} 
\;\leq\;
\frac{4}{\lambda_n (\CSIR) - \lambda_{n+1} (\CSIR)} || \mE_{21} ||_2
\end{equation}
where $\mE = \CSIR - \hCSIR$ and $\mE_{21} = \mB^T \mE \mA$. We wish to show that the condition for Corollary 8.1.11 holds with high probability for sufficiently large $N$. Theorem 2.6 from~\cite{Stewart90} states that for any $\tau > 0$
\begin{equation}
\label{eq:Cheb_ineq}
\mathbb{P} \left( \, ||\mE||_F \leq \tau \, \sqrt{\Exp{||\mE||_F^2}} \, \right)
\;\geq\;
1 - \frac{1}{\tau^2}
\end{equation}
which is derived from the Chebyshev inequality. Choose $\tau_*$ to be large such that $1/\tau_*^2$ is arbitrarily close to zero. Equation \eqref{eq:exp_E} states that $\Exp{||\mE||_F^2} = \sO (\Nrmin^{-1})$. Additionally, $\Nrmin \rightarrow \infty$ as $N \rightarrow \infty$ since $\omega_r > 0$ for each $r = 1,\dots,R$. Therefore, there exists $N_*$ such that when $N > N_*$
\begin{equation}
\tau_* \, \sqrt{\Exp{||\mE||_F^2}}
\;\leq\;
\frac{\lambda_n (\CSIR) - \lambda_{n+1} (\CSIR)}{5}
\end{equation}
Combining this with \eqref{eq:Cheb_ineq} implies that
\begin{equation}
\mathbb{P} \left( \, ||\mE||_F \leq \frac{\lambda_n (\CSIR) - \lambda_{n+1} (\CSIR)}{5} \, \right)
\;\geq\;
1 - \frac{1}{\tau_*^2} 
\end{equation}
when $N > N_*$. Since $||\cdot||_2 \leq ||\cdot||_F$,
\begin{equation}
\mathbb{P} \left( \, ||\mE||_2 \leq \frac{\lambda_n (\CSIR) - \lambda_{n+1} (\CSIR)}{5} \, \right) \geq 1 - \frac{1}{\tau_*^2} .
\end{equation}
By Corollary 8.1.11 in~\cite{Golub96},
\begin{equation}
\begin{aligned}
\dist{\mA}{\hat{\mA}} &\leq \frac{4}{\lambda_n (\CSIR) - \lambda_{n+1} (\CSIR)} || \mE_{21} ||_2 \\
&\leq \frac{4}{\lambda_n (\CSIR) - \lambda_{n+1} (\CSIR)} ||\mE||_F \\
&\leq \frac{4}{\lambda_n (\CSIR) - \lambda_{n+1} (\CSIR)} \tau_* \sqrt{\Exp{||\mE||_F^2}} \\
&= \frac{1}{\lambda_n (\CSIR) - \lambda_{n+1} (\CSIR)} \sO (\Nrmin^{-1/2})
\end{aligned}
\end{equation}
with probability $1 - 1/\tau_*^2$ when $N > N_*$.
\end{proof}
Thus, the subspace error decays at a rate $N_{r_{\min}}^{-1/2}$ with high probability for sufficiently large $N$. Perhaps the more interesting result from Theorem \ref{thm:SIR_sub_converge} is the inverse relationship between the subspace error and the magnitude of the gap between the $n$th and $(n+1)$th eigenvalues. That is, a large gap between eigenvalues suggests a better estimate of the subspace for a fixed number of samples. We choose not to subsume this factor within the $\sO$ notation to emphasize the importance of the spectral gap $\lambda_n-\lambda_{n+1}$ on SIR's ability to estimate the column span of $\mA$. One insight gained from explicitly revealing this part of the constant in the big-$\mathcal{O}$ notation is that the errors in the subspaces do not necessarily decay as a $n$ increases. In fact, if the eigenvalues plateau as a function of $n$ as is common in real data sets, then errors in the higher dimensional subspace estimates increase. 


\subsection{SAVE for ridge recovery}
\label{subsec:SAVE_for_RR}

We next consider Algorithm \ref{alg:SAVE} in the context of ridge recovery. The $\CAVE$ matrix from \eqref{eq:CAVE_reg} can be expressed as an integral with respect to the push-forward measure $\gamma$,
\begin{equation}
\label{eq:CAVE}
\CAVE
\;=\;
\int \left( \mI - \mSigma (y) \right)^2 \, d \gamma (y) ,
\end{equation}
where the conditional covariance is
\begin{equation}
\label{eq:cond_cov}
\mSigma (y)
\;=\;
\int \left( \vx - \vmu(y) \right) \, \left( \vx - \vmu(y) \right)^T \, d \sigma_{\vx|y} (\vx) .
\end{equation}
We use the same slicing function $h$ from \eqref{eq:hy_r} to enable approximation of \eqref{eq:CAVE} and \eqref{eq:cond_cov}. Considering ridge recovery of the sliced output leads to the matrix
\begin{equation}
\CSAVE \;=\; \sum_{r=1}^R \omega_r \, \left( \mI - \mSigma_r \right)^2 ,
\end{equation}
where $\omega_r$ is the probability mass function from \eqref{eq:slice_density} and 
\begin{equation}
\mSigma_r
\;=\;
\int \left( \vx - \vmu_r \right) \, \left( \vx - \vmu_r \right)^T \, d \sigma_{\vx|r} (\vx) .
\end{equation}
By containment of the central subspace, we know that 
\begin{equation}
\colspan{\CSAVE} \;\subseteq\; \sS_{h,\rho} \;\subseteq\; \sS_{f,\rho} .
\end{equation}
Additionally, we can interpret $\CSAVE$ as a Riemann sum approximation of $\CAVE$ using a similar argument as in the previous section for $CSIR$ and $CIR$. Algorithm \ref{alg:SAVE} computes a sample approximation of $\CSAVE$ (denoted by $\hCSAVE$) using given predictor/response pairs $\{ [ \, \vx_i^T \, , \, y_i \, ] \}$ for $i = 1,\dots,N$. For deterministic functions, we again use Monte Carlo sampling to construct an analogous set of input/output pairs. These pairs are then used in the SAVE algorithm to address the ridge recovery problem. The matrix $\hCSAVE$ is random, as are its eigenvalues and eigenspaces. Therefore, the convergence analysis of SAVE for ridge recovery is probabilistic.

The following theorem shows the rate of mean-squared convergence of the eigenvalues of $\hCSAVE$.
\begin{theorem}
\label{thm:SAVE_eig_converge}
Assume that Algorithm \ref{alg:SAVE} has been applied to the data set $\{ [ \, \vx_i^T \, , \, f(\vx_i) \, ] \}$, with $i=1,\dots,N$, where the $\vx_i$ are drawn independently according to $\rho$. Then, for $k = 1,\dots,m$,
\begin{equation}
\Exp{\left( \lambda_k (\CSAVE) - \lambda_k (\hCSAVE) \right)^2} \;=\; \sO (N^{-1}_{r_{\min}})
\end{equation}
where $\lambda_k (\cdot)$ denotes the $k$th eigenvalue of the given matrix.
\end{theorem}

\begin{proof}
Algorithm \ref{alg:SAVE} approximates the population matrix
\begin{equation}
\label{eq:CSAVE2}
\CSAVE \;=\; \sum_{r=1}^R \omega_r \, \left( \mI - \mSigma_r \right)^2 ,
\end{equation}
by the sample matrix 
\begin{equation}
\label{eq:hCSAVE}
\hCSAVE \;=\; \sum_{r=1}^R \hat{\omega}_r \, \left( \mI - \hat{\mSigma}_r \right)^2 ,
\end{equation}
where $\hat{\omega}_r$ and $\hat{\mSigma}_r$ are sample estimates of $\omega_r$ and $\mSigma_r$, respectively. Define $\sI_r$ to be the set of indices such that $y_i \in J_r$, and let $N_r$ be the cardinality of $\sI_r$. In the SAVE algorithm, these estimates are
\begin{equation}
\hat{\omega}_r
\;=\;
\frac{N_r}{N}
\qquad \text{and} \qquad
\hat{\mSigma}_r
\;=\;
\frac{1}{N_r - 1} \sum_{i \in \sI_r} \left( \vx_i - \hat{\vmu}_r \right) \left( \vx_i - \hat{\vmu}_r \right)^T
\end{equation}
where $\hat{\vmu}_r$ is the sample estimate of the average from \eqref{eq:sample_slice_average}. We may rewrite $\hat{\omega}_r$ and $\hat{\mSigma}_r$ as
\begin{equation}
\hat{\omega}_r
\;=\;
\frac{1}{N} \sum_{i=1}^N \chi \left( y_i \in J_r \right)
\quad \text{and} \quad
\hat{\mSigma}_r
\;=\;
\frac{1}{N_r - 1} \sum_{i \in \sI_r} \vx_i \, \vx_i^T - \frac{N_r}{N_r - 1} \hat{\vmu}_r \, \hat{\vmu}_r^T 
\end{equation}
where $\chi \left( y_i \in J_r \right)$ is an indicator function which is 1 when $y_i \in J_r$ and 0 otherwise. 

Fix $r$ such that we are considering a single term of the summations in \eqref{eq:CSAVE2} and \eqref{eq:hCSAVE}. Let 
\begin{equation}
\begin{aligned}
\omega \;=\; \omega_r , \quad
\hat{\omega} \;=\; \hat{\omega}_r , \quad
\chi_i \;=\; \chi \left( y_i \in J_r \right) , \hspace{3.5em} \\
\vmu \;=\; \vmu_r , \quad
\hat{\vmu} \;=\; \hat{\vmu}_r , \quad
\mSigma \;=\; \mSigma_r , \quad \text{and} \quad
\hat{\mSigma} \;=\; \hat{\mSigma}_r .
\end{aligned}
\end{equation}
Assume without loss of generality that $\sI_r = \{ 1 , \dots , N_r \}$. This simplifies notation as we compute the mean squared error.

To compute the mean squared error, we focus on the computation of a single element in $\hCSAVE$. To do this, we move the sample index to the superscript and let the subscript denote the vector element similar to proof of Theorem \ref{thm:SIR_eig_converge}. Additionally, we let $(\cdot)_{ij}$ denote the $ij$th element of the given matrix. Thus, for $1 \leq i , j \leq m$,
\begin{equation} \label{eq:SAVE_exp_of_prod1}
\begin{aligned}
\Exp{\hat{\omega} \left( \delta_{ij} - (\hat{\mSigma})_{ij} \right)^2}
\;&=\;
\mathbb{E} \left[ \left( \frac{1}{N} \sum_{k_1=1}^N \chi^{k_1} \right) 
\left( \delta_{ij} -  \left( \frac{1}{N_r - 1} \sum_{k_2=1}^{N_r} x_i^{k_2} x_j^{k_2} \right. \right. \right. \dots \\
&\hspace{11em} 
\left. \left. \left. - \left( \frac{1}{N_r} \sum_{k_3=1}^{N_r} x_i^{k_3} \right) \left( \frac{1}{N_r} \sum_{k_4=1}^{N_r} x_j^{k_4} \right) \right) \right)^2 \right] \\
\;&=\;
\frac{\delta_{ij}}{N} \sum_{k_1=1}^N \Exp{\chi^{k_1}} - \frac{2 \delta_{ij}}{N (N_r - 1)} \sum_{k_1=1}^N \sum_{k_2=1}^{N_r} \Exp{\chi^{k_1} x_i^{k_2} x_j^{k_2}} \dots \\
&\hspace{1.5em}
+ \frac{2 \delta_{ij}}{N N_r^2} \sum_{k_1=1}^N \sum_{k_2=1}^{N_r} \sum_{k_3=1}^{N_r} \Exp{\chi^{k_1} x_i^{k_2} x_j^{k_3}} \dots \\
&\hspace{1.5em}
+ \frac{1}{N (N_r - 1)^2} \sum_{k_1=1}^N \sum_{k_2=1}^{N_r} \sum_{k_3=1}^{N_r} \Exp{\chi^{k_1} x_i^{k_2} x_j^{k_2} x_i^{k_3} x_j^{k_3}} \dots \\
&\hspace{1.5em}
- \frac{2}{N N_r^2 (N_r - 1)} \sum_{k_1=1}^N \sum_{k_2=1}^{N_r} \sum_{k_3=1}^{N_r} \sum_{k_4=1}^{N_r} \Exp{\chi^{k_1} x_i^{k_2} x_j^{k_2} x_i^{k_3} x_j^{k_4}} \dots \\
&\hspace{1.5em}
+ \frac{1}{N N_r^4} \sum_{k_1=1}^N \sum_{k_2=1}^{N_r} \sum_{k_3=1}^{N_r} \sum_{k_4=1}^{N_r} \sum_{k_5=1}^{N_r} \Exp{\chi^{k_1} x_i^{k_2} x_j^{k_3} x_i^{k_4} x_j^{k_5}} .
\end{aligned}
\end{equation}
Equations \eqref{eq:K_multi-index_set} and \eqref{eq:pq_tensor} allow us to rewrite \eqref{eq:SAVE_exp_of_prod1} as sums over various tensors. Since $r$ is fixed, we drop the argument from the notation of \eqref{eq:K_multi-index_set} and \eqref{eq:pq_tensor}. Thus,
\begin{equation}
\begin{aligned}
\Exp{\hat{\omega} \left( \delta_{ij} - (\hat{\mSigma})_{ij} \right)^2}
\;&=\;
\frac{\delta_{ij}}{N} \sum_{k_1=1}^N \Exp{\chi^{k_1}}
- \frac{2 \delta_{ij}}{N (N_r - 1)} \sum_{\vk \in \sK^{1,1}} \mB^{1,1}_{\vk} \dots \\
&\hspace{0.5em}
+ \frac{2 \delta_{ij}}{N N_r^2} \sum_{\vk \in \sK^{1,2}} \mB^{1,2}_{\vk}
+ \frac{1}{N (N_r - 1)^2} \sum_{\vk \in \sK^{1,2}} \mB^{1,2}_{\vk} \dots \\
&\hspace{0.5em}
- \frac{2}{N N_r^2 (N_r - 1)} \sum_{\vk \in \sK^{1,3}} \mB^{1,3}_{\vk}
+ \frac{1}{N N_r^4} \sum_{\vk \in \sK^{1,4}} \mB^{1,4}_{\vk} .
\end{aligned}
\end{equation}
By Lemma \ref{lem:general_B_sum},
\begin{equation} \label{eq:SAVE_exp_one_r}
\begin{aligned}
\Exp{\hat{\omega} \left( \delta_{ij} - (\hat{\mSigma})_{ij} \right)^2}
\;&=\;
\frac{\delta_{ij}}{N} \sum_{k_1=1}^N \Exp{\chi^{k_1}}
- \frac{2 \delta_{ij}}{N (N_r - 1)} \sum_{\vk \in \sK^{1,1}} \mB^{1,1}_{\vk} \dots \\
&\hspace{0.5em}
+ \frac{2 \delta_{ij}}{N N_r^2} \sum_{\vk \in \sK^{1,2}} \mB^{1,2}_{\vk}
+ \frac{1}{N (N_r - 1)^2} \sum_{\vk \in \sK^{1,2}} \mB^{1,2}_{\vk} \dots \\
&\hspace{0.5em}
- \frac{2}{N N_r^2 (N_r - 1)} \sum_{\vk \in \sK^{1,3}} \mB^{1,3}_{\vk}
+ \frac{1}{N N_r^4} \sum_{\vk \in \sK^{1,4}} \mB^{1,4}_{\vk} \\
\;&=\;
\frac{\delta_{ij}}{N} \left[ N \, \Exp{\chi} \right]
- \frac{2 \delta_{ij}}{N (N_r - 1)} \left[ N N_r \, \Exp{\chi} \Exp{x_i x_j} + \sO (N + N_r) \right] \dots \\
&\hspace{0.5em}
+ \frac{2 \delta_{ij}}{N N_r^2} \left[ N N_r^2 \, \Exp{\chi} \Exp{x_i} \Exp{x_j} + \sO (N N_r + N_r^2) \right] \dots \\
&\hspace{0.5em}
+ \frac{1}{N (N_r - 1)^2} \left[ N N_r^2 \, \Exp{\chi} \Exp{x_i x_j}^2 + \sO (N N_r + N_r^2) \right] \dots \\
&\hspace{0.5em}
- \frac{2}{N N_r^2 (N_r - 1)} \left[ N N_r^3 \, \Exp{\chi} \Exp{x_i x_j} \Exp{x_i} \Exp{x_j} + \sO (N N_r^2 + N_r^3) \right] \dots \\
&\hspace{0.5em}
+ \frac{1}{N N_r^4} \left[ N N_r^4 \, \Exp{\chi} \Exp{x_i}^2 \Exp{x_j}^2 + \sO (N N_r^3 + N_r^4) \right] \\ 
\;&=\;
\delta_{ij} \Exp{\chi}
- 2 \delta_{ij} \Exp{\chi} \Exp{x_i x_j}
+ 2 \delta_{ij} \Exp{\chi} \Exp{x_i} \Exp{x_j} \dots \\
&\hspace{0.5em}
+ \Exp{\chi} \Exp{x_i x_j}^2
- 2 \, \Exp{\chi} \Exp{x_i x_j} \Exp{x_i} \Exp{x_j} \dots \\
&\hspace{0.5em}
+ \Exp{\chi} \Exp{x_i}^2 \Exp{x_j}^2
+ \sO (N_r^{-1}) \\
\;&=\;
\Exp{\chi} \left( \delta_{ij} - \left( \Exp{x_i x_j} - \Exp{x_i} \Exp{x_j} \right) \right)^2 + \sO (N_r^{-1}) \\
\;&=\;
\omega \left( \delta_{ij} - (\hat{\mSigma})_{ij} \right)^2 + \sO (N_r^{-1}) .
\end{aligned}
\end{equation}
Next, we compute the variance using the property
\begin{equation}
\Var{\hat{\omega} \left( \delta_{ij} - (\hat{\mSigma})_{ij} \right)^2} = \Exp{\left( \hat{\omega} \left( \delta_{ij} - (\hat{\mSigma})_{ij} \right)^2 \right)^2} - \Exp{\hat{\omega} \left( \delta_{ij} - (\hat{\mSigma})_{ij} \right)^2}^2 .
\end{equation}
To find $\Exp{\left( \hat{\omega} \left( \delta_{ij} - (\hat{\mSigma})_{ij} \right)^2 \right)^2}$,
\begin{equation} \label{eq:SAVE_exp_of_prod2}
\begin{aligned}
\Exp{\left( \hat{\omega} \left( \delta_{ij} - (\hat{\mSigma})_{ij} \right)^2 \right)^2}
\;&=\;
\mathbb{E} \left[ \left( \frac{1}{N} \sum_{k_1=1}^N \chi^{k_1} \right)^2
\left( \delta_{ij} -  \left( \frac{1}{N_r - 1} \sum_{k_2=1}^{N_r} x_i^{k_2} x_j^{k_2} \right. \right. \right. \dots \\
&\hspace{9em}
\left. \left. \left. - \left( \frac{1}{N_r} \sum_{k_3=1}^{N_r} x_i^{k_3} \right) \left( \frac{1}{N_r} \sum_{k_4=1}^{N_r} x_j^{k_4} \right) \right) \right)^4 \right] .
\end{aligned}
\end{equation}
By expanding, we can rewrite \eqref{eq:SAVE_exp_of_prod2} in a form which can be simplified using the tensor summation notation from \eqref{eq:K_multi-index_set} and \eqref{eq:pq_tensor}. We then apply Lemma \ref{lem:general_B_sum} to simplify these summations and obtain
\begin{equation}
\Exp{\left( \hat{\omega} \left( \delta_{ij} - (\hat{\mSigma})_{ij} \right)^2 \right)^2}
\;=\;
\left( \omega \left( \delta_{ij} - (\mSigma)_{ij} \right)^2 \right)^2 + \sO (N_r^{-1}) .
\end{equation}
Thus,
\begin{equation} \label{eq:SAVE_var_one_r}
\begin{aligned}
\Var{\hat{\omega} \left( \delta_{ij} - (\hat{\mSigma})_{ij} \right)^2}
\;&=\;
\Exp{\left( \hat{\omega} \left( \delta_{ij} - (\hat{\mSigma})_{ij} \right)^2 \right)^2}
- \Exp{\hat{\omega} \left( \delta_{ij} - (\hat{\mSigma})_{ij} \right)^2}^2 \\
\;&=\;
\left( \left( \omega \left( \delta_{ij} - (\mSigma)_{ij} \right)^2 \right)^2 + \sO (N_r^{-1}) \right) \\
&\hspace{7em} 
- \left( \omega \left( \delta_{ij} - (\mSigma)_{ij} \right)^2  + \sO (N_r^{-1}) \right)^2 \\
\;&=\;
\sO (N_r^{-1}) .
\end{aligned}
\end{equation}
Similar to the proof of Theorem \ref{thm:SIR_eig_converge}, we extend \eqref{eq:SAVE_exp_one_r} and \eqref{eq:SAVE_var_one_r} from one slice to the total summation to obtain the expectation and variance for a single element of $\hCSAVE$,
\begin{equation} \label{eq:SAVE_exp_and_var}
\Exp{(\hCSAVE)_{ij}} \;=\; (\CSAVE)_{ij} + \sO \left( \Nrmin^{-1} \right), 
\qquad
\Var{(\hCSAVE)_{ij}} \;=\; \sO \left( \Nrmin^{-1} \right)
\end{equation}
where $\Nrmin$ from \eqref{eq:N_r_min} denotes the minimum number of samples in any one slice. From \eqref{eq:SAVE_exp_and_var}, the mean squared error for a single element of the SAVE matrix is
\begin{equation}
\MSE{(\hCSAVE)_{ij}} \;=\; \sO (\Nrmin^{-1}) .
\end{equation}
Similar to the proof of Theorem \ref{thm:SIR_eig_converge}, we use this mean squared error to obtain
\begin{equation}
\Exp{||\mE||_F^2} \;=\; \sO (\Nrmin^{-1})
\end{equation}
where $\mE = \CSAVE - \hCSAVE$. Combining this result with Corollary 8.1.6 in~\cite{Golub96} yields the desired result,
\begin{equation}
\Exp{ \left( \lambda_k (\CSAVE) - \lambda_k (\hCSAVE) \right)^2} \;=\; \sO \left( \Nrmin^{-1} \right) ,
\end{equation}
as required.
\end{proof}

Next, we examine the convergence of the subspaces Algorithm \ref{alg:SAVE} produces, where the subspace distance is from \eqref{eq:sub_dist}.
\begin{theorem}
\label{thm:SAVE_sub_converge}
Assume the same conditions from Theorem \ref{thm:SAVE_eig_converge}. Then, for sufficiently large $N$,
\begin{equation}
\dist{\mA}{\hat{\mA}} \;=\; 
\frac{1}{\lambda_n (\CSAVE) - \lambda_{n+1} (\CSAVE)}\;\sO (N^{-1/2}_{r_{\min}})
\end{equation}
with high probability.
\end{theorem}

\begin{proof}
In the proof of Theorem \ref{thm:SAVE_eig_converge}, we showed that 
\begin{equation}
\Exp{||\mE||_F^2} \;=\; \sO (\Nrmin^{-1})
\end{equation}
where $\mE = \CSAVE - \hCSAVE$. Given this result, the proof for Theorem \ref{thm:SAVE_sub_converge} is identical to the proof for Theorem \ref{thm:SIR_sub_converge}.
\end{proof}

The subspace error for Algorithm \ref{alg:SAVE} decays like $N_{r_{\min}}^{-1/2}$ with high probability for sufficiently large $N$. Similar to the estimated SIR subspace from Algorithm \ref{alg:SIR}, the error depends inversely on the eigenvalue gap. If the gap between the $n$th and $(n+1)$th eigenvalues is large, then the error in the estimated $n$-dimensional subspace is small for a fixed number of samples. This fact gives insight into the errors in estimated subspaces of different dimensions. In particular, if the true eigenvalues plateau as $n$ increases---as is common in practice---then the estimated subspace errors increase with increasing $n$.

\section{Numerical results}
\label{sec:numericalresults}

We examine the performance of SIR and SAVE for ridge recovery in three problems that exhibit known ridge structure. The first two problems are multivariate quadratic functions, and the third problem is a physically-motivated test problem from magnetohydrodynamics. We provide several graphics in the numerical study of these problems including plots of estimated eigenvalues, eigenvalue errors, and subspace errors. One graphic is especially useful for visualizing the structure of the function relative to its central subspace: the sufficient summary plot~\cite{Cook98}. Sufficient summary plots show $y$ versus $\mA^T \vx$, where $\mA$ contains a basis for the central subspace. Algorithms \ref{alg:SIR} and \ref{alg:SAVE} produce eigenvectors that span an approximation of the SIR and SAVE subspaces, respectively. We use these eigenvectors to construct the one- or two-dimensional inputs for sufficient summary plots.

We emphasize that the goal of these numerical experiments is to verify the expected numerical behavior (i.e., asymptotic convergence) of SIR and SAVE for ridge recovery. We do not compare (e.g., cost versus accuracy) with other dimension reduction methods that might be used for ridge recovery. The point of this paper is not comparison with other methods; it is to understand theoretically how the inverse regression methods SIR and SAVE can be applied and interpreted in the context of deterministic approximation. A comprehensive numerical comparison with other methods should be supported by theoretical comparison, and these comparisons are beyond the scope of the present manuscript. Moreover, although the need for dimension reduction is broadly motivated by expensive computer models, we do not study the SIR and SAVE performance on expensive models. Since our goal is to verify asymptotic rates, we use cheap functions that can be evaluated many times. 

The python code used to generate the figures found throughout this section are available at \url{https://bitbucket.org/aglaws/inverse-regression-for-ridge-recovery}. The scripts require the dev branch of the Python Active-subspaces Utility Library~\cite{joss2016}, which can be found at \url{https://github.com/paulcon/active_subspaces/tree/dev}. We note that the convergence analysis from Sections \ref{subsec:SIR_for_RR} and \ref{subsec:SAVE_for_RR} assumes a fixed slicing of the observed $y$ range. However, Algorithms \ref{alg:SIR} and \ref{alg:SAVE} are implemented in the Utility Library using an adaptive slicing approach that attempts to maximize $\Nrmin$ for a given set of data. This is done as a heuristic technique for reducing eigenvalue and subspace errors.

\subsection{One-dimensional quadratic function}
\label{ssec:numex1}
Let $\vx \in \mathbb{R}^{10}$, and let $\rho$ be a standard multivariate Gaussian measure. Define the function
\begin{equation} \label{eq:1d_quad}
y \;=\; f(\vx) \;=\; \left( \vb^T \vx \right)^2,
\end{equation}
where $\vb \in \mathbb{R}^{10}$ is a constant vector. The span of $\vb$ is the central subspace. First, we attempt to uncover the central subspace using SIR (Algorithm \ref{alg:SIR}), which is known to fail for functions symmetric about $\vx = \mzero$~\cite{Cook91}; Figure \ref{fig:quadratic1_SIR} confirms this failure. In fact, $\CIR$ for this problem is zero since the conditional expectation of $\vx$ for any value of $y$ is zero. Figure \ref{fig:quadratic1_SIR_evals} shows that all estimated eigenvalues of the SIR matrix are nearly zero as expected. Figure \ref{fig:quadratic1_SIR_1d_ssp} is a one-dimensional sufficient summary plot of $y$ against $\hat{\vw}_1^T \vx$, where $\hat{\vw}_1$ denotes the normalized eigenvector associated with the largest eigenvalue of $\hCSIR$ from Algorithm \ref{alg:SIR}. If (i) the central subspace is one-dimensional (as in this case) and (ii) the chosen SDR algorithm correctly identifies the one basis vector, then the sufficient summary plot will show a univariate relationship between the linear combination of input evaluations and the associated outputs. Due to the symmetry in the quadratic function, SIR fails to recover the basis vector. However, it should be noted that if \eqref{eq:1d_quad} had the form $y = f(\vx) = ( \vb^T \vx + \vc)^2$ for a constant $\vc \neq \mzero$, then SIR would not suffer these issues. 

\begin{figure}[!ht]
\centering
\subfloat[Eigenvalues of $\hCSIR$ from \eqref{eq:hDSIR}]{
\label{fig:quadratic1_SIR_evals}
\includegraphics[width=0.4\textwidth]{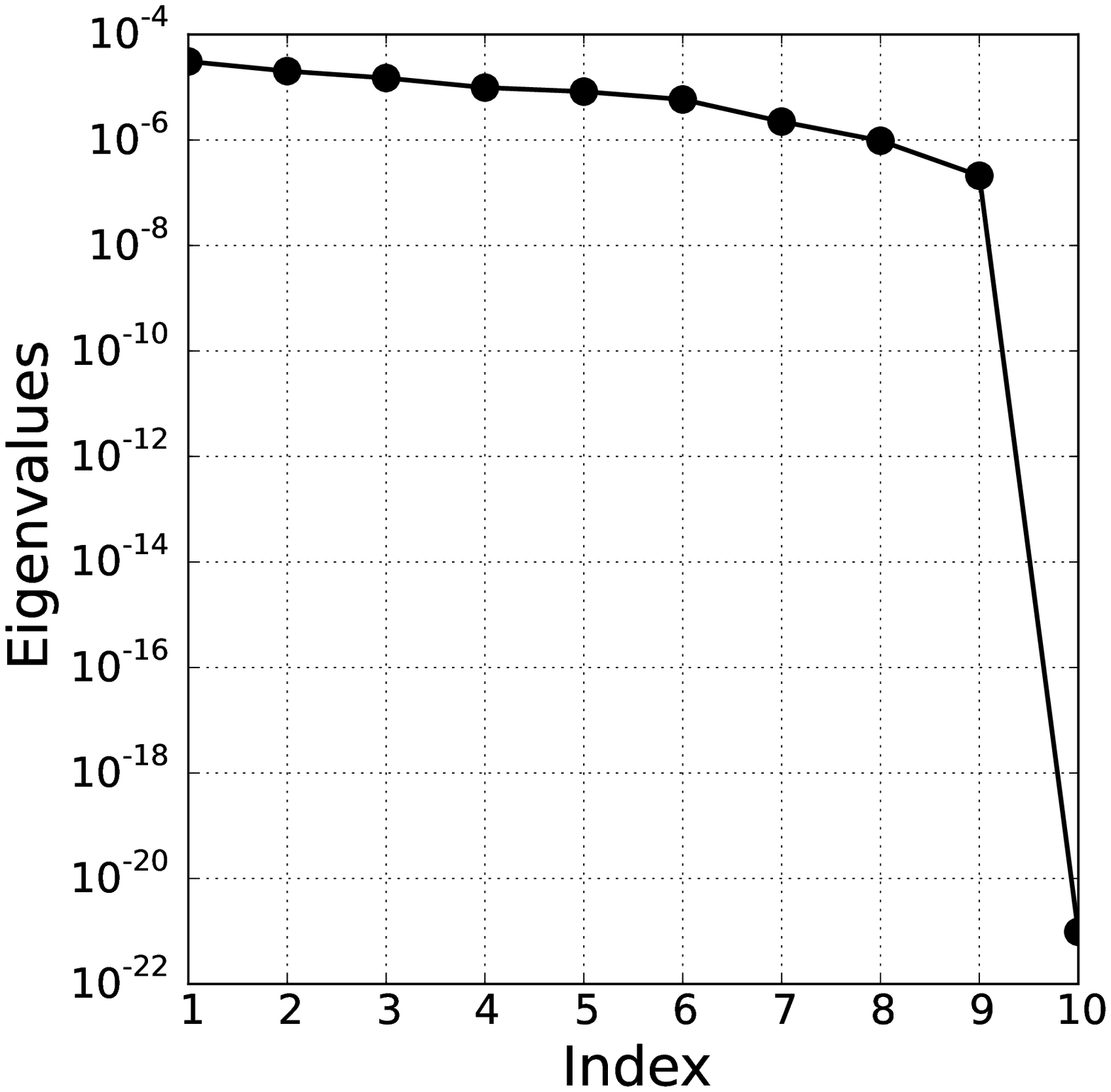}
}
\hfil
\subfloat[Sufficient summary plot for SIR]{
\label{fig:quadratic1_SIR_1d_ssp}
\includegraphics[width=0.4\textwidth]{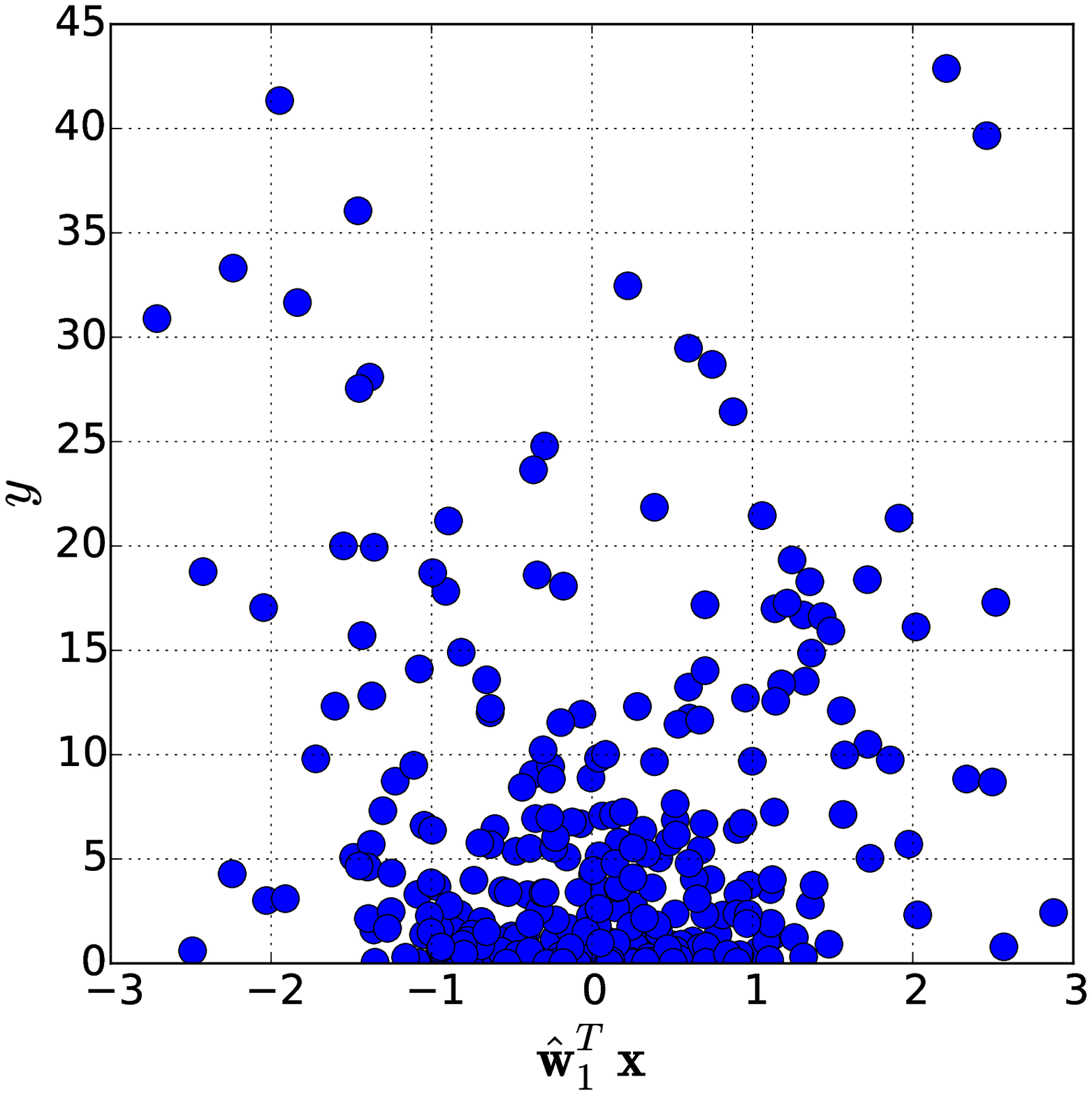}
}
\caption{As expected, SIR fails to recover the ridge direction $\vb$ in the function \eqref{eq:1d_quad}.}
\label{fig:quadratic1_SIR}
\end{figure}

Figure \ref{fig:quadratic1_SAVE} shows results from applying SAVE (Algorithm \ref{alg:SAVE}) to the quadratic function \eqref{eq:1d_quad}. Figure \ref{fig:quadratic1_SAVE_evals} shows the eigenvalues of $\hCSAVE$ from Algorithm \ref{alg:SAVE}. Note the large gap between the first and second eigenvalues, which suggests that the SAVE subspace is one-dimensional. Figure \ref{fig:quadratic1_SAVE_1d_ssp} shows the sufficient summary plot using the first eigenvector $\hvw_1$ from Algorithm \ref{alg:SAVE}, which reveals the true univariate quadratic relationship between $\hvw_1^T\vx$ and $y$.

\begin{figure}[!ht]
\centering
\subfloat[Eigenvalues of $\hCSAVE$ from \eqref{eq:hDSAVE}]{
\label{fig:quadratic1_SAVE_evals}
\includegraphics[width=0.4\textwidth]{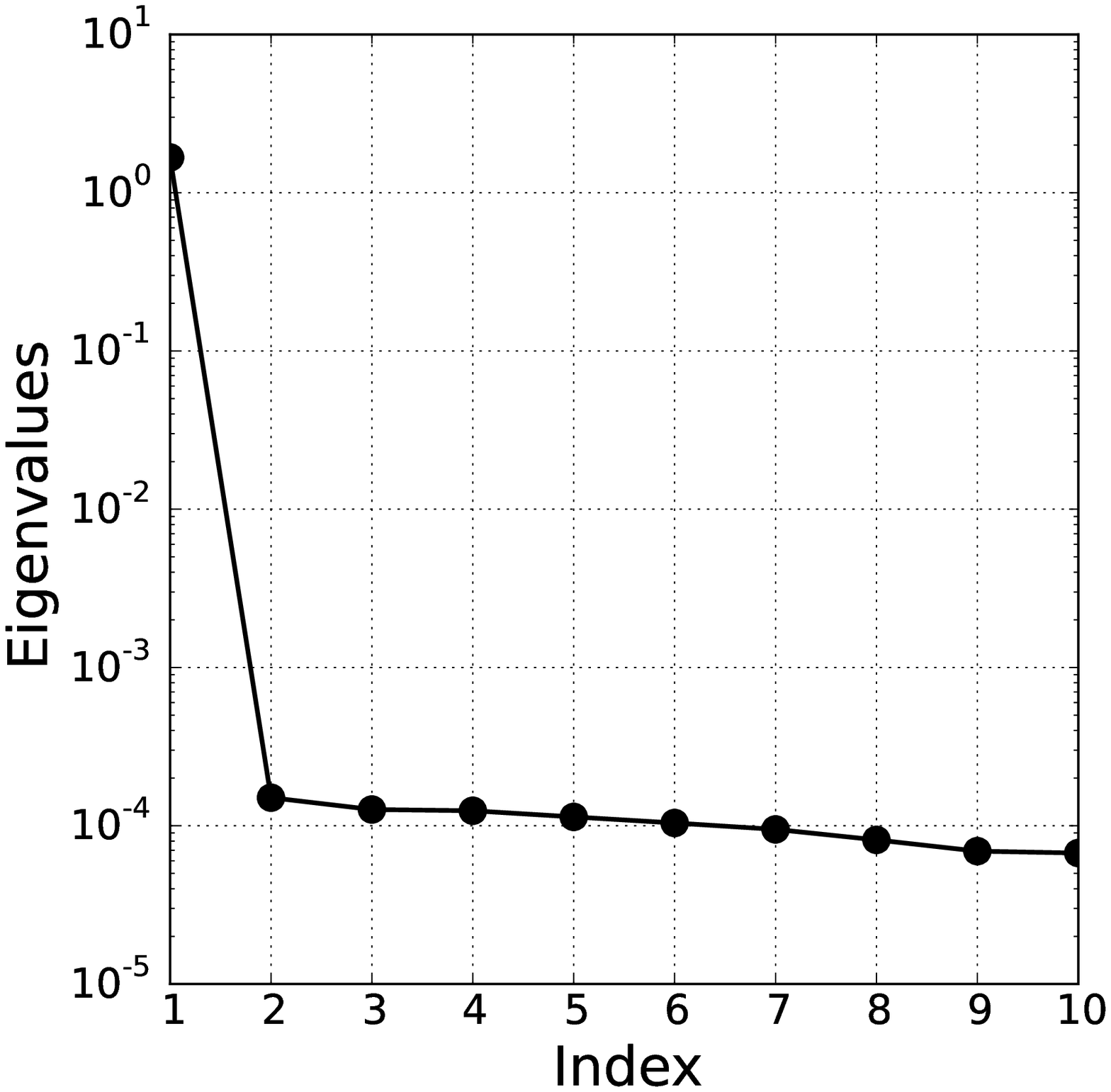}
}
\hfil
\subfloat[Sufficient summary plot for SAVE]{
\label{fig:quadratic1_SAVE_1d_ssp}
\includegraphics[width=0.4\textwidth]{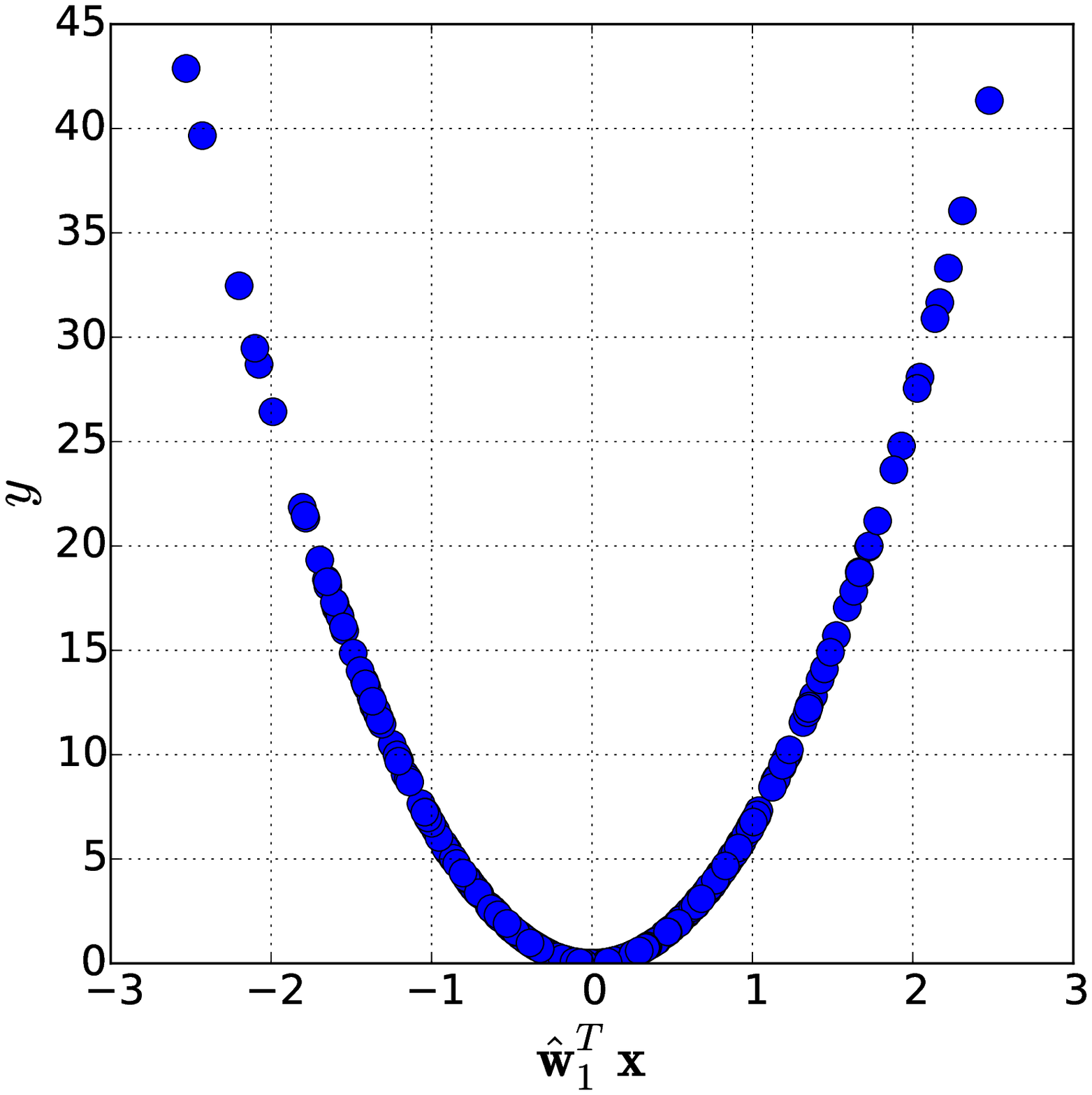}
}
\caption{SAVE recovers the ridge direction $\vb$ in the function \eqref{eq:1d_quad}.}
\label{fig:quadratic1_SAVE}
\end{figure}

\subsection{Three-dimensional quadratic function}

Next, we numerically study the convergence properties of SIR and SAVE using a more complex quadratic function. Let $\vx \in \mathbb{R}^{10}$, and let $\rho$ be the standard multivariate Gaussian measure. Define the function
\begin{equation} \label{eq:3d_quad}
y \;=\; f(\vx) \;=\; \vx^T \mB \mB^T \vx + \vb^T \vx,
\end{equation}
where $\mB \in \mathbb{R}^{10 \times 2}$ and $\vb \in \mathbb{R}^{10}$ with $\vb\not\in\colspan{\mB}$. We expect better results with SIR compared to the example in \eqref{eq:1d_quad} since \eqref{eq:3d_quad} is not symmetric about $\vx = \mzero$. Figure \ref{fig:quadratic2_SIR_evals} shows the eigenvalues of $\hCSIR$ decay; note the gap between the third and fourth eigenvalues. Figure \ref{fig:quadratic2_SIR_eval_errs} shows the maximum squared eigenvalue error normalized by the largest eigenvalue,
\begin{equation}
\max_{1 \leq i \leq m} \frac{\left( \lambda_i (\hCSIR) - \lambda_i (\CSIR) \right)^2}{\lambda_1 (\CSIR)^2} ,
\end{equation}
for increasing numbers of samples in 10 independent trials. We estimate the true eigenvalues using SIR with $10^7$ samples. The mean squared error decays at a rate slightly faster than the $\sO( N^{-1} )$ from Theorem \ref{thm:SIR_eig_converge}. The improvement is likely attributed to the adaptive slicing procedure discussed at the beginning of this section. Figure \ref{fig:quadratic2_SIR_sub_errs} shows the error in the estimated three-dimensional SIR subspace (see \eqref{eq:sub_dist}) for increasing numbers of samples in 10 independent trials. We use $10^7$ samples to create a surrogate \emph{truth} subspace for the convergence study. The subspace errors decrease asymptotically at a rate of approximately $\sO( N^{-1/2} )$, which agrees with Theorem \ref{thm:SIR_sub_converge}.

\begin{figure}[!ht]
\centering
\subfloat[Eigenvalues of $\hCSIR$ from \eqref{eq:hDSIR}]{
\label{fig:quadratic2_SIR_evals}
\includegraphics[width=0.4\textwidth]{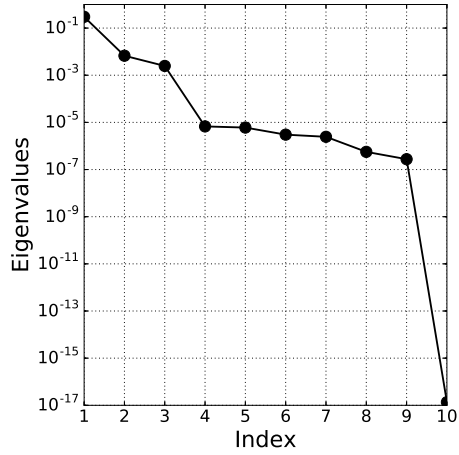}
} \\
\subfloat[Maximum squared eigenvalue error, SIR]{
\label{fig:quadratic2_SIR_eval_errs}
\includegraphics[width=0.4\textwidth]{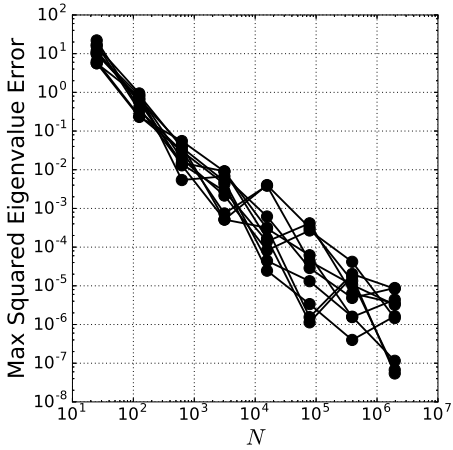}
}
\hfil
\subfloat[SIR subspace errors for $n = 3$]{
\label{fig:quadratic2_SIR_sub_errs}
\includegraphics[width=0.4\textwidth]{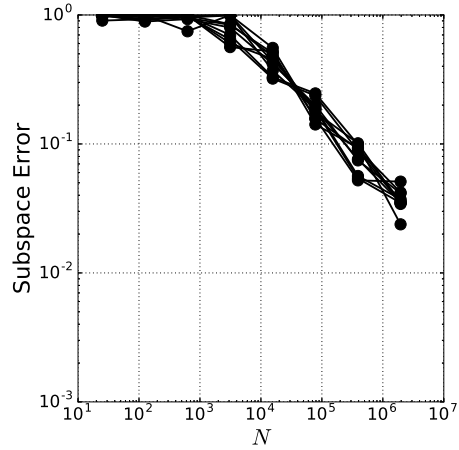}
}
\caption{Eigenvalues, eigenvalue errors, and subspace errors for SIR applied to \eqref{eq:3d_quad}. The error decreases with increasing samples consistent with the convergence theory in section \ref{subsec:SIR_for_RR}.}
\label{fig:quadratic2_SIR}
\end{figure}

Figure \ref{fig:quadratic2_SAVE} shows the results of a similar convergence study using SAVE (Algorithm \ref{alg:SAVE}). The eigenvalues of $\hCSAVE$ from \eqref{eq:hDSAVE} are shown in Figure \ref{fig:quadratic2_SAVE_evals}. Note the large gap between the third and fourth eigenvalues, which is consistent with the three-dimensional central subspace in $f(\vx)$ from \eqref{eq:3d_quad}. Figures \ref{fig:quadratic2_SAVE_eval_errs} and \ref{fig:quadratic2_SAVE_sub_errs} show the maximum squared eigenvalue error and the subspace error for $n = 3$, respectively. The eigenvalue error again decays at a faster rate than expected in Theorem \ref{thm:SAVE_eig_converge}, possibly due to the adaptive slicing implemented in the code. The subspace error decays consistently according to Theorem \ref{thm:SAVE_sub_converge}.

\begin{figure}[!ht]
\centering
\subfloat[Eigenvalues of $\hCSAVE$ from \eqref{eq:hDSAVE}]{
\label{fig:quadratic2_SAVE_evals}
\includegraphics[width=0.4\textwidth]{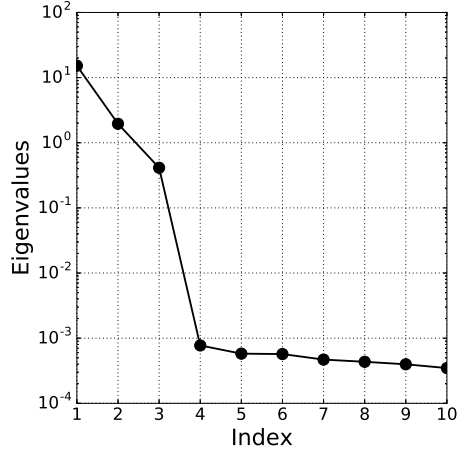}
} \\
\subfloat[Maximum squared eigenvalue error, SAVE]{
\label{fig:quadratic2_SAVE_eval_errs}
\includegraphics[width=0.4\textwidth]{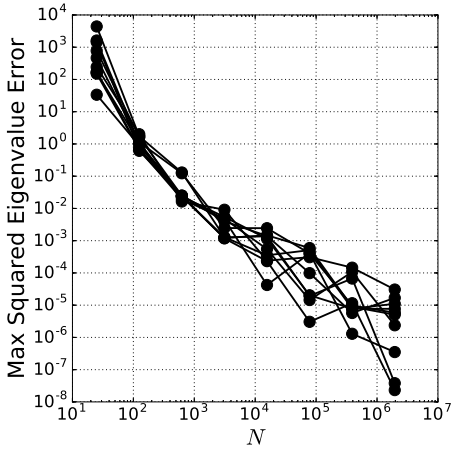}
}
\hfil
\subfloat[SAVE subspace errors for $n = 3$]{
\label{fig:quadratic2_SAVE_sub_errs}
\includegraphics[width=0.4\textwidth]{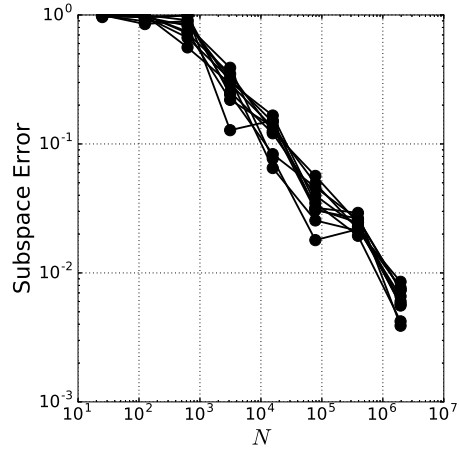}
}
\caption{Eigenvalues, eigenvalue errors, and subspace errors for SAVE applied to \eqref{eq:3d_quad}. The error decreases with increasing samples consistent with the convergence theory in section \ref{subsec:SAVE_for_RR}.}
\label{fig:quadratic2_SAVE}
\end{figure}

\subsection{Hartmann problem}

We use the following model as a test case for parameter space dimension reduction methods in recent work; the problem set up here closely follows our previous development~\cite{Glaws17b}. The Hartmann problem is a standard problem in magnetohydrodynamics (MHD) that models the flow of an electrically-charged plasma in the presence of a uniform magnetic field~\cite{Cowling57}. The flow occurs along an infinite channel between two parallel plates separated by distance $2 \ell$. The applied magnetic field is perpendicular to the flow direction and acts as a resistive force on the flow velocity. At the same time, the movement of the fluid induces a magnetic field along the direction of the flow. Figure \ref{fig:hartmann_prob} contains a diagram of this problem.

\begin{figure}[ht]
\begin{center}
\includegraphics[width=0.5\textwidth]{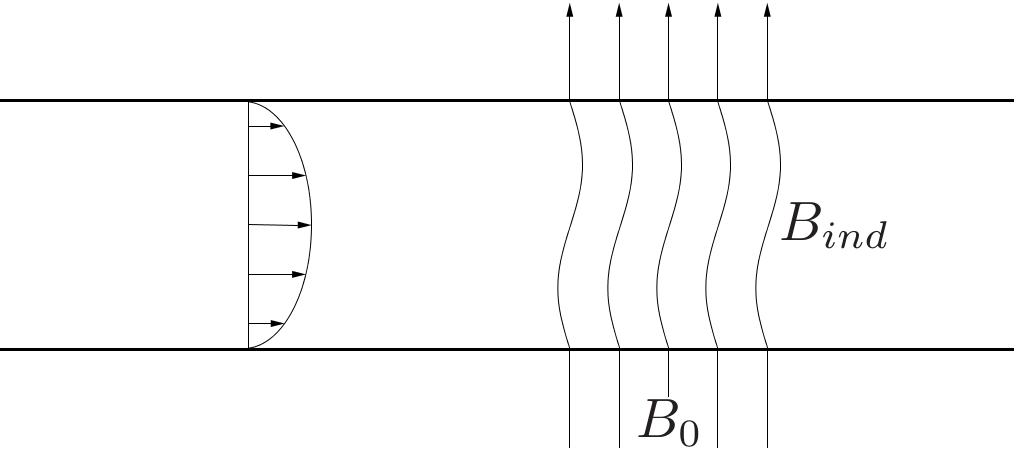}
\end{center}
\caption{The Hartmann problem studies the flow of an ionized fluid between two parallel plates. A magnetic field applied perpendicular to the flow direction acts as a resistive force to the fluid flow. Simultaneously, the fluid induces a magnetic field in the direction of the flow.}
\label{fig:hartmann_prob}
\end{figure}

The inputs to the Hartmann model are fluid viscosity $\mu$, fluid density $\rho$, applied pressure gradient $\partial p_0 / \partial x$ (where the derivative is with respect to the flow field's spatial coordinate), resistivity $\eta$, and applied magnetic field $B_0$. We collect these inputs into a vector,
\begin{equation} \label{eq:hart_inputs}
\vx \;=\; \bmat{ \mu & \rho & \frac{\partial p_0}{\partial x} & \eta & B_0 }^T.
\end{equation}
We consider one output of interest: the total induced magnetic field,
\begin{equation}
\label{eq:Bind}
\Bind (\vx) = \frac{\partial p_0}{\partial x} \frac{\ell \mu_0}{2 B_0} \left( 1 - 2 \frac{\sqrt{\eta \mu}}{B_0 \ell} \, \text{tanh} \left( \frac{B_0 \ell}{2 \sqrt{\eta \mu}} \right) \right) .
\end{equation}
This function is not a ridge function of $\vx$. However, it has been shown that many physical laws can be expressed as ridge functions by considering a log transform of the inputs~\cite{Constantine16}. For this reason, we apply SIR and SAVE ridge recovery to $\Bind$ as a function of the logarithms of the inputs from \eqref{eq:hart_inputs}. The log-transformed inputs are drawn from a multivariate Gaussian with
\begin{equation}
\vmu = \bmat{ -2.25 \\ 1 \\ 0.3 \\ 0.3 \\ -0.75 }, \qquad 
\mSigma = \bmat{ 0.15 & & & & \\ & 0.25 & & & \\ & & 0.25 & & \\ & & & 0.25 & \\ & & & & 0.25} .
\end{equation}
Figure \ref{fig:hartmann_Bind_SIR} shows the results of applying SIR (Algorithm \ref{alg:SIR}) to the Hartmann model for the induced magnetic field $\Bind$. The eigenvalues of $\hCSIR$ from \eqref{eq:hDSIR} with bootstrap ranges are shown in Figure \ref{fig:hartmann_Bind_SIR_evals}. Large gaps appear after the first and second eigenvalues which suggest possible two-dimensional ridge structure. In fact, the induced magnetic field admits a two-dimensional central subspace relative to the log-inputs~\cite{Glaws17b}. Figures \ref{fig:hartmann_Bind_SIR_1d_ssp} and \ref{fig:hartmann_Bind_SIR_2d_ssp} contain one- and two-dimensional sufficient summary plots of $\Bind$ against the $\hvw_1^T \vx$ and $\hvw_2^T \vx$, where $\hvw_1$ and $\hvw_2$ are the first two eigenvectors of $\hCSIR$. We see a strong one-dimensional relationship. However, the two-dimensional sufficient summary plot shows slight curvature with changes in $\hvw_2^T\vx$. These results suggest that ridge-like structure may be discovered using the SIR algorithm in some cases. Figure \ref{fig:hartmann_Bind_SIR_suberr} shows the subspace errors as a function of the subspace dimension. Recall from Theorem \ref{thm:SIR_sub_converge} that the subspace error depends inversely on the eigenvalue gap. The largest eigenvalue gap occurs between the first and second eigenvalues, which is consistent with the smallest subspace error for $n=1$.

\begin{figure}[!ht]
\centering
\subfloat[Eigenvalues of $\hCSIR$ with bootstrap ranges]{
\label{fig:hartmann_Bind_SIR_evals}
\includegraphics[width=0.4\textwidth]{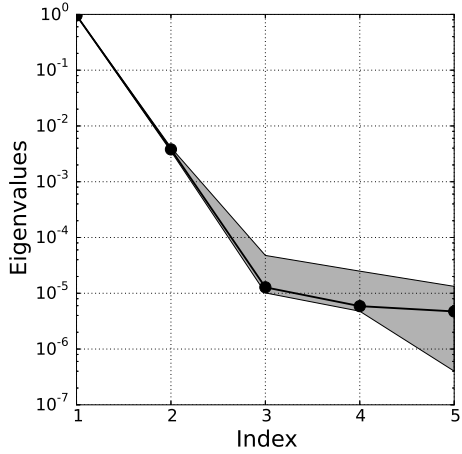}
}
\hfil
\subfloat[Subspace errors from $\hCSIR$]{
\label{fig:hartmann_Bind_SIR_suberr}
\includegraphics[width=0.4\textwidth]{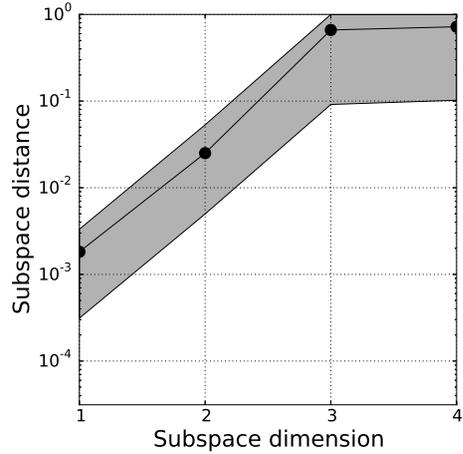}
} \\
\subfloat[One-dimensional SIR summary plot for $\Bind$ from \eqref{eq:Bind}]{
\label{fig:hartmann_Bind_SIR_1d_ssp}
\includegraphics[width=0.4\textwidth]{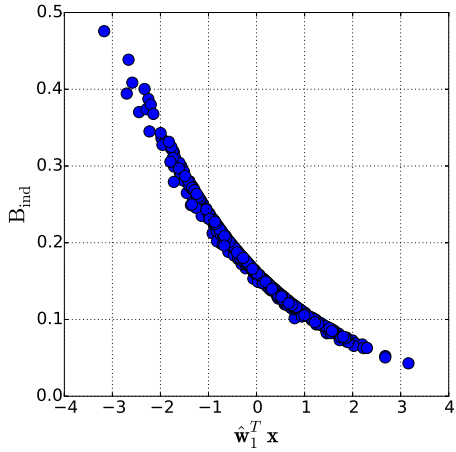}
}
\hfil
\subfloat[Two-dimensional SIR summary plot for $\Bind$ from \eqref{eq:Bind}]{
\label{fig:hartmann_Bind_SIR_2d_ssp}
\includegraphics[width=0.4\textwidth]{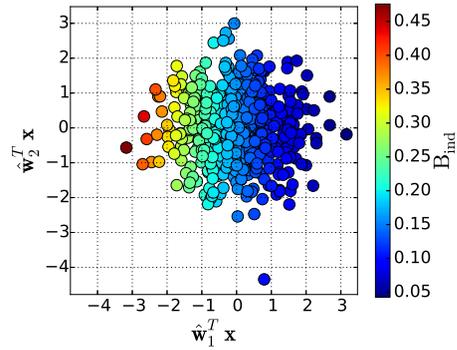}
}
\caption{Eigenvalues with bootstrap ranges, estimated subspace errors, and sufficient summary plots for SIR (Algorithm \ref{alg:SIR}) applied to $\Bind$ from \eqref{eq:Bind}.}
\label{fig:hartmann_Bind_SIR}
\end{figure}

We perform the same numerical studies for SAVE applied to $\Bind$. Figure \ref{fig:hartmann_Bind_SAVE_evals} shows the eigenvalues of the $\hCSAVE$ from \eqref{eq:hDSAVE} for the induced magnetic field $\Bind$ from \eqref{eq:Bind}. Note the large gaps after the first and second eigenvalues. These gaps are consistent with the subspace errors in Figure \ref{fig:hartmann_Bind_SAVE_suberr}, where the one- and two-dimensional subspace estimates have the smallest errors. Figures \ref{fig:hartmann_Bind_SAVE_1d_ssp} and \ref{fig:hartmann_Bind_SAVE_2d_ssp} contain sufficient summary plots for $\hvw_1^T\vx$ and $\hvw_2^T\vx$, where $\hvw_1$ and $\hvw_2$ are the first two eigenvectors from $\hCSAVE$ in \eqref{eq:hDSAVE}.

\begin{figure}[!ht]
\centering
\subfloat[Eigenvalues of $\hCSAVE$ with bootstrap ranges]{
\label{fig:hartmann_Bind_SAVE_evals}
\includegraphics[width=0.4\textwidth]{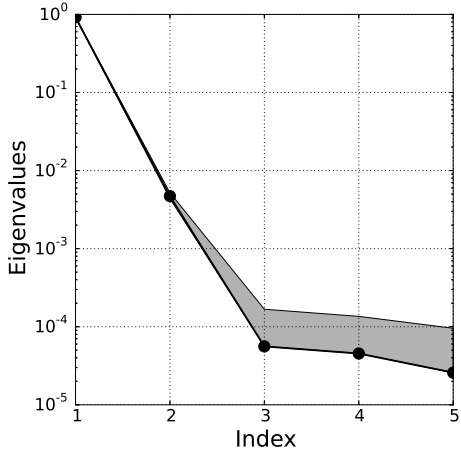}
}
\hfil
\subfloat[Subspace errors from $\hCSAVE$]{
\label{fig:hartmann_Bind_SAVE_suberr}
\includegraphics[width=0.4\textwidth]{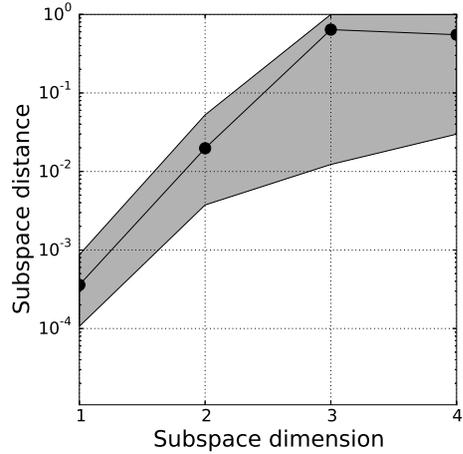}
} \\
\subfloat[One-dimensional SAVE summary plot for $\Bind$ from \eqref{eq:Bind}]{
\label{fig:hartmann_Bind_SAVE_1d_ssp}
\includegraphics[width=0.4\textwidth]{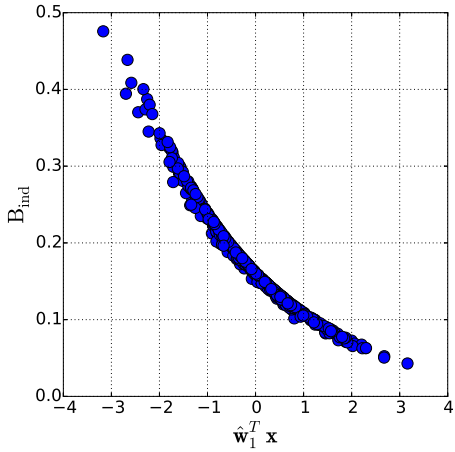}
}
\hfil
\subfloat[Two-dimensional SAVE summary plot for $\Bind$ from \eqref{eq:Bind}]{
\label{fig:hartmann_Bind_SAVE_2d_ssp}
\includegraphics[width=0.4\textwidth]{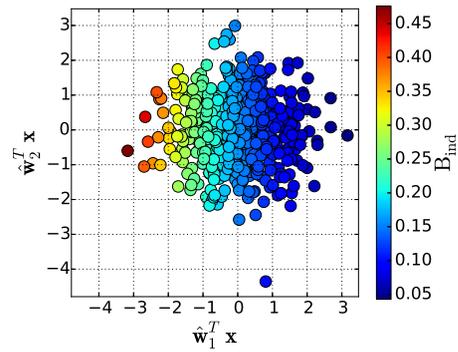}
}
\caption{Eigenvalues with bootstrap ranges, estimated subspace errors, and sufficient summary plots for SAVE (Algorithm \ref{alg:SAVE}) applied to $\Bind$ from \eqref{eq:Bind}.}
\label{fig:hartmann_Bind_SAVE}
\end{figure}

\section{Summary and conclusion}
\label{sec:conclusion}

We study the numerical behavior of two inverse regression methods---sliced inverse regression and sliced average variance estimation---applied to the ridge recovery problem. We reinterpret these methods as Monte Carlo approximations of matrices of integrals. This allows us to study the numerical convergence of the algorithms. The mean-squared error of the estimated SIR and SAVE eigenvalues converges at a rate inversely proportional to $\Nrmin$---the smallest number of samples in any slice. The distance between the estimated subspace and the true subspace decays like $\Nrmin^{-1/2}$ for both SIR and SAVE, although the error in the estimated $n$-dimensional subspace depends inversely on the gap between the $n$th and $(n+1)$th eigenvalues. Finally, we illustrate the results of the numerical analysis by applying the methods to several test problems, including two quadratic functions and a simplified MHD model.

\section*{Acknowledgments}
\noindent We thank Luis Tenorio and Don Estep for helpful discussions and insights regarding the work presented in this paper.

\bibliographystyle{siamplain}
\bibliography{deterministicSDR}

\end{document}